\documentclass[12pt,a4paper]{amsart}
\usepackage{amsmath}
\usepackage{amsthm}
\usepackage{amssymb}
\usepackage[english]{babel}
\usepackage{verbatim}
\usepackage[shortlabels]{enumitem}
\usepackage{tikz-cd}
\usepackage{multirow}

\newtheoremstyle{mio}%
	{}{} % spazio sopra e sotto%
	{\itshape}{} % corpo del testo, indentation
	{\bfseries}{.}{ } % titolo del teorema: tipo di testo, divisore, spaziatura
	{#1 #2\thmnote{\mdseries~ #3}} % formattazione nota
\theoremstyle{mio}
\newtheorem{teor}{Theorem}[section]
\newtheorem{cor}[teor]{Corollary}
\newtheorem{prop}[teor]{Proposition}

\newtheorem{defin}[teor]{Definition}

\theoremstyle{definition}
\newtheorem{ex}[teor]{Example}
\newtheorem{oss}[teor]{Remark}

\renewcommand{\star}{\ast}

\newcommand{\inssubmod}{\mathbf{F}}

\newcommand{\insstar}{\mathrm{Star}}
\newcommand{\inssmstar}{\mathrm{(S)Star}}
\newcommand{\inssemistar}{\mathrm{SStar}}
\newcommand{\inssemisupp}[1]{\mathrm{SStar}^{#1}}
\newcommand{\inssmsupp}[1]{\mathrm{(S)Star}^{#1}}
\newcommand{\insfstar}{\mathrm{FStar}}

\newcommand{\Jac}{\mathrm{Jac}}

\newcommand{\SkOver}{\mathrm{SkOver}}

\newcommand{\valut}{\mathbf{v}}

\DeclareMathOperator{\supp}{supp}

\DeclareMathOperator{\Spec}{Spec}
\newcommand{\hiSpec}{\mathrm{Spec_{hi}}}
\newcommand{\boldnu}{\boldsymbol{\nu}}
\newcommand{\tildhom}{\widetilde{\hom}}

\DeclareMathOperator{\Max}{Max}
\DeclareMathOperator{\Over}{Over}
\newcommand{\insfracid}{\mathcal{F}}

\newcommand{\insZ}{\mathbb{Z}}
\newcommand{\inv}[1]{\frac{1}{#1}}
\newcommand{\insN}{\mathbb{N}}
\newcommand{\inN}{\in\mathbb{N}}
\newcommand{\insQ}{\mathbb{Q}}

\title[Semistar operations on semilocal Pr\"ufer domains]{The sets of star and semistar operations on semilocal Pr\"ufer domains}
\author{Dario Spirito}
\address{Dipartimento di Matematica e Fisica, Universit\`a degli Studi ``Roma Tre'', Roma, Italy}
\email{spirito@mat.uniroma3.it}

\begin{document}

\begin{abstract}
We study the sets of semistar and star operation on a semilocal Pr\"ufer domain, with an emphasis on which properties of the domain are enough to determine them. In particular, we show that these sets depend chiefly on the properties of the spectrum and of some localizations of the domain; we also show that, if the domain is $h$-local, the number of semistar operations grows as a polynomial in the number of semistar operations of its localizations.
\end{abstract}

\maketitle

\section{Introduction}
Starting from the works of Krull \cite{krull_breitage_I-II}, Gilmer \cite[Chapter 32]{gilmer} and Okabe and Matsuda \cite{okabe-matsuda}, the study of star and semistar operations has usually followed the route of studying properties holding for some classes of these operations, or of some particular cases: for example, studying the properties of stable, spectral \cite{anderson_two_2000,anderson_intersections_2005,localizing-semistar} or eab operations (see e.g. \cite{fontana_loper-eab} and \cite[Section 4]{fifolo_transactions}), or studying the $t$- \cite{samuel_factoriel,bouvier_zaf_1988} or the $b$-operation \cite{swanson_huneke}.

More recently, there has been interest in studying these closures from a global perspective, that is, in studying the properties of the whole set: for example, studying a natural topology on the set of semistar operations \cite{topological-cons,spettrali-eab}, or studying the relationship between semistar and semiprime operations \cite{epstein-corresp}. In particular, Houston, Mimouni and Park have been interested in the study of the cardinality of the set of star operations in the Noetherian setting \cite{houston_noeth-starfinite,starnoeth_resinfinito}, as well as in the integrally closed case (with special interest in the case of Pr\"ufer domains) \cite{twostar,hmp_finite,hmp-overprufer}: in \cite{hmp_finite} they showed that there is a strong link between the spectrum of a semilocal Pr\"ufer domain $D$ and the number of star operations on $D$, while in \cite[Theorem 4.3]{hmp-overprufer} they calculated the number of star operations when the spectrum of $D$ is Y-shaped. With different methods, Elliott showed that the structure of the set of semistar operations on a Dedekind domain $D$ (in particular, its cardinality) depends only on the number of maximal ideals of $D$ \cite{elliott-dedekind}.

In this paper, we deepen this study, linking it to the concept of Jaffard family (whose tie with star operations was established in \cite{starloc}) and extending it to semistar operations. In particular, we focus on which information about a Pr\"ufer semilocal domain $D$ is sufficient to determine the sets $\inssemistar(D)$ and $\insstar(D)$ of, respectively, semistar and star operations. We show in Theorem \ref{teor:semistar} that $\inssemistar(D)$ can be determined by joining some geometric data (the spectrum of $D$, or more precisely the homeomorphically irreducible tree underlying $\Spec(D)$) and some algebraic data (the set of semistar operations on some valuation rings of the form $D_P/QD_P$). We then show (Theorem \ref{teor:(semi)star}) that, to determine $\insstar(D)$, we must also add some information about the maximal ideals of $D$ (namely, if they are principal). We also show (Corollary \ref{cor:semistar-hloc}) that the cardinality of $\inssemistar(D)$, when $D$ is an $h$-local domain with $n$ maximal ideals, is a polynomial of degree $n\cdot 2^{n-1}$ in the number of semistar operations on the localizations $D_P$.

\section{Notation and preliminaries}
\subsection{Closures and semistar operations}
Let $(\mathcal{P},\leq)$ be a partially ordered set. A \emph{closure operation} on $\mathcal{P}$ is a map $c:\mathcal{P}\longrightarrow\mathcal{P}$ such that:
\begin{enumerate}
\item $c$ is \emph{extensive}: $x\leq c(x)$ for every $x\in\mathcal{P}$;
\item $c$ is \emph{order-preserving}: if $x\leq y$, then $c(x)\leq c(y)$;
\item $c$ is \emph{idempotent}: $c(c(x))=c(x)$ for every $x\in\mathcal{P}$.
\end{enumerate}
If $x\in\mathcal{P}$ is such that $x=c(x)$, then $x$ is said to be \emph{$c$-closed}.

Let now $D$ be an integral domain with quotient field $K$; let $\inssubmod(D)$ be the set of $D$-submodule of $K$, and let $\insfracid(D)$ be the set of \emph{fractional ideals} of $D$, i.e., of the $I\in\inssubmod(D)$ such that $xI\subseteq D$ for some $x\in K$, $x\neq 0$.

If $\star:I\mapsto I^\star$ is a closure operation on $\inssubmod(D)$ or $\insfracid(D)$, let $(\mathbf{S})$ be the following property:

\smallskip
\indent $(\mathbf{S})$: $x\cdot I^\star=(xI)^\star$ for every $x\in K$ and every $I$ where $\star$ is defined.

\smallskip

This property is usually used to define the following three classes of closure operations:
\begin{itemize}
\item \emph{semistar operations} are closure operations on $\inssubmod(D)$ with property $(\mathbf{S})$;
\item \emph{(semi)star operations} are semistar operations $\star$ such that $D=D^\star$;
\item \emph{star operations} are closure operations $\star$ on $\insfracid(D)$ with property $(\mathbf{S})$ and such that $D=D^\star$.
\end{itemize}
We denote the sets of these closures, respectively, as $\inssemistar(D)$, $\inssmstar(D)$ and $\insstar(D)$.

We shall need a fourth class of closure operations:
\begin{defin}
A \emph{fractional star operation} on $D$ is a closure operation on $\insfracid(D)$ with property $(\mathbf{S})$. We denote their set by $\insfstar(D)$.
\end{defin}

\begin{comment}
\begin{table}
\begin{tabular}[c]{|c|c|c|c|}
\hline
\multicolumn{2}{|c|}{\multirow{2}{*}{}} & \multicolumn{2}{c|}{\textbf{Set of definition}}\\
\cline{3-4}
\multicolumn{2}{|c|}{\multirow{2}{*}{}} & $\insfracid(D)$ &$\inssubmod(D)$\\
\hline
\multirow{2}{*}{\textbf{Image of $D$}} & $D^\star=D$ & star operations & (semi)star operations\\
\cline{2-4}
& $D^\star$ arbitrary & fractional star operations & semistar operations\\
\hline
\end{tabular}
\smallskip
\caption{Closure operations on domains.}
\end{table}
\end{comment}

These four sets are all partially ordered, with $\star_1\leq\star_2$ if $I^{\star_1}\subseteq I^{\star_2}$ for every $I$ (belonging to $\insfracid(D)$ or $\inssubmod(D)$, according to the case).

The identity map, $I\mapsto I$, is a closure operation, and it is denoted by $d$ both in the semistar and in the star setting.

\subsection{Localizations of star operations}
Let $\star\in\insstar(D)$ and let $T$ be a flat overring of $D$. Then, $\star$ is said to be \emph{extendable} to $T$ if the map
\begin{equation*}%\label{eq:definloc}
\begin{aligned}
\star_T\colon \insfracid(T) & \longrightarrow \insfracid(T)\\
IT & \longmapsto I^\star T
\end{aligned}
\end{equation*}
is well-defined (where $I$ is a fractional ideal of $D$) \cite[Definition 3.1]{starloc}. In this case, $\star_T$ is a star operation. The same definition can be given in the case of fractional star operations and semistar operations; it works well in the former case, but poorly in the latter \cite[Remark 5.12]{starloc}.

\subsection{Jaffard families and localizations}\label{sect:jaff}
Let $D$ be an integral domain with quotient field $K$. An \emph{overring} of $D$ is a ring between $D$ and $K$; the set of overrings of $D$ is denoted by $\Over(D)$. A set $\Theta$ of overrings of $D$ is a \emph{Jaffard family} of $D$ if the following properties hold \cite[Proposition 4.3]{starloc}:
\begin{itemize}
\item $I=\bigcap\{IT\mid T\in\Theta\}$ for every ideal $I$ of $D$;
\item $\Theta$ is \emph{locally finite} (i.e., for every $x\in K$, $x$ is not invertible in at most a finite number of $T\in\Theta$);
\item $K\notin\Theta$;
\item $TS=K$ for every $T\neq S$ in $\Theta$.
\end{itemize}
(This is only one possible definition; see \cite[beginning of Section 6.3 and Theorem 6.3.5]{fontana_factoring} for two different characterizations.) In particular, if $\Theta$ is a Jaffard family of $D$, then \cite[Theorem 6.3.1]{fontana_factoring}:
\begin{itemize}
\item for every prime ideal of $P$ there is exactly one $T\in\Theta$ such that $PT\neq T$; in particular, $\Theta$ induces a partition on $\Max(D)$;
\item $I=\bigcap\{IT\mid T\in\Theta\}$ for every $I\in\inssubmod(D)$;
\item every $T\in\Theta$ is flat over $D$.
\end{itemize}

If $\Theta$ is a Jaffard family of $D$, then, for every $T\in\Theta$, each star operation on $D$ is extendable to $T$; moreover, the map
\begin{equation*}
\begin{aligned}
\lambda_\Theta\colon\insstar(D) & \longrightarrow\prod_{T\in\Theta}\insstar(T) \\
\star & \longmapsto (\star_T)_{T\in\Theta},
\end{aligned}
\end{equation*}
is an order isomorphism \cite[Theorem 5.4]{starloc}. An inspection of the proof of this result shows that the same reasoning also gives a bijection from $\insfstar(D)$ to $\prod\{\insfstar(T)\mid T\in\Theta\}$. On the other hand, the analogue of this result does not hold for semistar operations \cite[Remark 5.12]{starloc}.

\subsection{The standard decomposition}
Let $D$ be a Pr\"ufer domain. Two maximal ideals $M$ and $N$ are \emph{dependent} if there is a nonzero prime ideal $P\subseteq M\cap N$, or equivalently if $D_MD_N\neq K$. Since the spectrum of a Pr\"ufer domain is a tree, dependence is an equivalence relation. Let $\{\Delta_\lambda\mid\lambda\in\Lambda\}$ be the set of equivalence classes of this relation, and define $T_\lambda:=\bigcap\{D_P\mid P\in\Delta_\lambda\}$; we call the set $\{T_\lambda\mid\lambda\in\Lambda\}$ the \emph{standard decomposition} of $D$. If $D$ is semilocal, or more generally if $\Max(D)$ is a Noetherian space, then the standard decomposition of $D$ is a Jaffard family of $D$ \cite[Proposition 6.2]{starloc}.

\subsection{Semistar operations and quotients}
Let $D$ be a Pr\"ufer domain, and suppose there is a nonzero prime ideal $P$ contained in the Jacobson radical $\Jac(D)$ of $D$. Then, $PD_P=P$, and so $D_P$ is a fractional ideal of $D$; it follows that every overring of $D$, except the quotient field $K$, is a fractional ideal of $D$. Hence, in this case $\insfstar(D)=\inssemistar(D)\setminus\{\wedge_{\{K\}}\}$ and $\inssmstar(D)=\insstar(D)$, where $\wedge_{\{K\}}$ is the semistar operation sending every nonzero $I\in\inssubmod(D)$ to $K$.

Let $\varphi:D_P\longrightarrow D_P/P=:k$ be the quotient map; then, $A:=D/P$ is a subring of $k$ with quotient field $k$. Let $\star\in\inssemistar(D)$ be a semistar operation such that $P=P^\star$. Then, $D_P=(P:P)$ is also $\star$-closed, and thus, for every $I\in\inssubmod(D)$ such that $P\subseteq I\subseteq D_P$, we have $P\subseteq I^\star\subseteq D_P$. Following \cite{fontana-park} and \cite{hmp_finite}, we define a semistar operation $\star_\varphi$ on $D/P$ by
\begin{equation*}
I^{\star_\varphi}:=\varphi\left(\varphi^{-1}(I)^\star\right)\quad\text{for every~}I\in\inssubmod(D/P).
\end{equation*}
Conversely, if $\sharp\in\inssemistar(D/P)$, then we can define a map $\sharp^\varphi$ from $\inssubmod(D)$ to itself in the following way: let $\valut_P$ be the valuation relative to $D_P$, and let $I\in\inssubmod(D)$. Then, we set $I^{\sharp^\varphi}:=I$ if $\valut_P(I)$ has no infimum in $\valut_P(K)$; otherwise, if $\valut_P(\alpha)=\inf\valut_P(I)$, then $P\subseteq\alpha^{-1}I\subseteq D_P$, and we put
\begin{equation*}
I^{\sharp^\varphi}:=\alpha\cdot\varphi^{-1}\left[(\varphi(\alpha I))^\sharp\right].
\end{equation*}

We have the following.
\begin{prop}\label{prop:cutting}
Let $D,P,A,\varphi$ as above; let $\Delta_1:=\{\star\in\inssemistar(D)\mid P=P^\star\}$ and $\Delta_2:=\{\star\in\inssemistar(D)\mid P\neq P^\star\}$. Then:
\begin{enumerate}[(a)]
\item\label{prop:cutting:Delta1} The maps
\begin{equation*}
\begin{aligned}
\Delta_1 & \longrightarrow \inssemistar(D/P)\\
\star & \longmapsto \star_\varphi
\end{aligned}
\quad\text{~and~}\quad
\begin{aligned}
\inssemistar(D/P) & \longrightarrow \Delta_1\\
\sharp & \longmapsto \sharp^\varphi
\end{aligned}
\end{equation*}
are well-defined order isomorphisms, inverses one of each other, that restricts to isomorphisms between $\inssmstar(D)=\insstar(D)$ and $\inssmstar(D/P)$.
\item\label{prop:cutting:Delta2} The map
\begin{equation*}
\begin{aligned}
\iota_P\colon\Delta_2 & \longrightarrow\inssemistar(D_P)\setminus\{d\} \\
\star & \longmapsto \star|_{\inssubmod(D_P)}
\end{aligned}
\end{equation*}
is a well-defined order isomorphism.
\item\label{prop:cutting:ordine} If $\star_1\in\Delta_1$ and $\star_2\in\Delta_2$ then $\star_1\leq\star_2$.
\end{enumerate}
\end{prop}
\begin{proof}
\ref{prop:cutting:Delta1} The proof is entirely analogous to the proof of \cite[Lemmas 2.3 and 2.4]{hmp_finite}.

\ref{prop:cutting:Delta2} It is clear that $\iota_P$ is well-defined and order-preserving; to see that is it bijective, it is enough to note that the map $\rho_P:\inssemistar(D_P)\longrightarrow\inssemistar(D)$ such that $I^{\rho_P(\star)}:=(ID_P)^\star$ is well-defined, sends $\inssemistar(D_P)\setminus\{d\}$ to $\Delta_2$, and it is the inverse of $\iota_P$.

\ref{prop:cutting:ordine} The overring $D_P$ is $\star_1$-closed for every $\star_1\in\Delta_1$; hence, $\star_1|_{\inssubmod(D_P)}$ is a (semi)star operation on $D_P$ which closes $P$. Being $D_P$ a valuation domain, this implies that $\star_1|_{\inssubmod(D_P)}$ is the identity; therefore, $I^{\star_1}\subseteq ID_P$ for every $I\in\inssubmod(D)$. But, if $\star_2\in\Delta_2$, then $\star_2=\rho_P(\iota_P(\star_2))$, so that $I^{\star_2}\supseteq ID_P$ for every $I$. Hence, $\star_1\leq\star_2$.
\end{proof}

\subsection{Product and sum of posets}\label{sect:posets}
Let $\mathcal{P}_1,\mathcal{P}_2$ be two partially ordered set. The \emph{product} of $\mathcal{P}_1$ and $\mathcal{P}_2$, denoted by $\mathcal{P}_1\times\mathcal{P}_2$, is the partial order on the Cartesian product such that $(x_1,y_1)\leq(x_2,y_2)$ if and only if $x_1\leq x_2$ and $y_1\leq y_2$.

The \emph{ordinal sum} of $\mathcal{P}_1$ and $\mathcal{P}_2$, denoted by $\mathcal{P}_1\oplus\mathcal{P}_2$ is the partial order on the disjoint union of $\mathcal{P}_1$ and $\mathcal{P}_2$ such that the order on each $\mathcal{P}_i$ is the same, while if $x\in\mathcal{P}_1$ and $y\in\mathcal{P}_2$ then $x\leq y$ \cite[Chapter 1, \textsection 8]{birkhoff_latticetheory}.

Under this terminology, Proposition \ref{prop:cutting} can be rewritten as saying that $\inssemistar(D)$ is isomorphic to the ordinal sum of $\inssemistar(D/P)$ and $\inssemistar(D_P)\setminus\{d\}$.

\subsection{Homeomorphically irreducible trees}\label{sect:Omef-irred}
Let $\mathcal{T}$ be a finite tree. Then, $\mathcal{T}$ is said to be \emph{homeomorphically irreducible} (or \emph{series-reduced}) if no vertex has valence 2 (where the \emph{valence} of $x$ is the number of element of $\mathcal{P}$ directly linked to $x$) \cite{treelike,harari-Omefirred}. When $\mathcal{T}$ is a rooted tree, we allow the root to have valence 2 (this is in contrast with the definition in \cite{harari-Omefirred} and \cite{treelike}, but is needed for our applications).

If $\mathcal{T}$ is a (possibly infinite) rooted tree, with root $r$, $\mathcal{T}$ has a natural structure of partially ordered set, where $x\leq y$ if the (unique) path from $r$ to $y$ passes through $x$. Call $x\in\mathcal{T}$ a \emph{branching point} if $x=r$ or if there is a family $\Delta\subseteq\mathcal{T}$ of pairwise incomparable elements such that $x\notin\Delta$ but $x$ is the infimum of $\Delta$; we say that $\mathcal{T}$ is homeomorphically irreducible if each element of $\mathcal{T}$ is a branching point. If $\mathcal{T}$ is finite, it is not hard to see that this definition coincides with the previous one.

Let $\mathcal{T}$ be a rooted tree. Then, the set of all branching points of $\mathcal{T}$ is an homeomorphically irreducible tree, which we call the \emph{underlying homeomorphycally irreducible tree} associated to $\mathcal{T}$.

\section{The support of a semistar operation}
In the paper, $D$ will always indicate a Pr\"ufer domain, and $K$ its quotient field. We shall study only semilocal Pr\"ufer domains, that is, domains with only a finite number of maximal ideals; while many definitions do make sense even in a more general setting, many results do not hold outside the semilocal case. In particular, the two results we shall continuously use are the existence of a standard decomposition $\Theta$ and the following Proposition \ref{prop:skover}.

\begin{defin}\label{def:skover}
Let $D$ be a semilocal Pr\"ufer domain, and let $\Theta$ be its standard decomposition. The \emph{skeleton of $\Over(D)$}, indicated by $\SkOver(D)$, is the set of all intersections of elements of $\Theta$.
\end{defin}

In particular, $\SkOver(D)$ contains $D$ (the intersection of all elements of $\Theta$) and the quotient field $K$ (the empty intersection), as well as the elements of $\Theta$. We note that the structure (as a partially ordered set) of $\SkOver(D)$ depends uniquely on the cardinality of $\Theta$, and that $\SkOver(D)$ is closed by intersections.

The main use of $\SkOver(D)$ passes though the following proposition, which can be seen as a variant of \cite[Lemma 4.2]{hmp_finite}. 

\begin{prop}\label{prop:skover}
Let $D$ be a semilocal Pr\"ufer domain. Then, $\inssubmod(D)$ is the disjoint union of $\insfracid(A)$, as $A$ ranges in $\SkOver(D)$.
\end{prop}
\begin{proof}
Let $\Theta$ be the standard decomposition of $D$, let $I\in\inssubmod(D)$, and consider the set $\supp(I):=\{T\in\Theta\mid IT\neq K\}$ (which we call the \emph{support} of $I$); we claim that $I$ is a fractional ideal of $A:=\bigcap\{T\mid T\in\supp(I)\}$.

Indeed, since $\Theta$ is a Jaffard family we have $I=\bigcap\{IT\mid T\in\Theta\}$. Moreover, we can throw away the elements of $\Theta$ outside the support, so that $I=\bigcap\{IT\mid T\in\supp(I)\}$; hence, $I$ is an $A$-module. Each $T\in\Theta$ is semilocal, and by the definition of the standard decomposition there is a nonzero prime ideal $P$ contained in the Jacobson radical $\Jac(T)$ of $T$. Then, $P=PT_P$; in particular, $pT_P\subseteq T$ for every $p\in P$, so that $T_P$ is a fractional ideal of $T$ and $(IT)T_P\neq K$. Since $T_P$ is a valuation domain, it follows that $IT$ is a fractional ideal of $T_P$, or equivalently $aIT\subseteq T_P$ for some $a\neq 0$. Hence, $apIT\subseteq T$ for any $p\in P$; choose one, and let $d_T:=ap$. Since $\supp(I)$ is finite, we can define $d$ as the product of such $d_T$; hence
\begin{equation*}
dI=d\bigcap_{T\in\Theta}IT=\bigcap_{T\in\Theta}dIT\subseteq\bigcap_{T\in\Theta}T=A.
\end{equation*}
Therefore, $I\in\insfracid(A)$, as claimed.

Suppose now that $\insfracid(A)\cap\insfracid(B)\neq\emptyset$ for some $A\neq B$ in $\SkOver(D)$. We can suppose that $A\subsetneq B$ (just substitute $A$ with $A\cap B$), and thus we can take $T\in\Theta$ containing $A$ but not $B$. Each overring of $A$ is flat over $D$, and $\supp(B)$ is finite; hence, by \cite[I.2.6, Proposition 6]{bourbaki_ac},
\begin{equation*}
BT=\left(\bigcap_{S\in\supp(B)}S\right)\cdot T=\bigcap_{S\in\supp(B)}ST=K.
\end{equation*}

Let now $I\in\insfracid(A)\cap\insfracid(B)$; then, for every $i\in I$, $i^{-1}I$ is a $B$-module containing $1$, and thus $B\subseteq i^{-1}I$. Since $i^{-1}I$ is also an $A$-fractional ideal, it means that $dB\subseteq A$ for some $d\neq 0$. Hence, $dB\subseteq T$, and so $dBT\subseteq TT=T$; however, $BT=K$, and thus we would have $dK\subseteq T$, a contradiction. Hence, the union is disjoint.
\end{proof}

\begin{oss}
~\begin{enumerate}
\item $\SkOver(D)$ is the unique subset of $\Over(D)$ which allows to split $\inssubmod(D)$ into sets of fractional ideals. Indeed, if $\inssubmod(D)=\bigsqcup\{\insfracid(A)\mid A\in\mathcal{A}\}$ for some other $\mathcal{A}$, then clearly $\mathcal{A}$ cannot properly contain $\SkOver(D)$, and thus there is a $B\in\SkOver(D)\setminus\mathcal{A}$. Thus, $B\in\insfracid(A)$ for some $A\in\mathcal{A}$, and $A\in\insfracid(B')$ for some $B'\in\SkOver(D)$; this means that $B\in\insfracid(B')$, which implies that $B=B'=A$. But, for any two overrings $R_1$ and $R_2$, $\insfracid(R_1)=\insfracid(R_2)$ implies $R_1=R_2$; hence $B\in\SkOver(D)$, a contradiction.
\item Proposition \ref{prop:skover} cannot be extended outside the semilocal case. For example, if $D=\insZ$, let $\mathbb{P}$ be the set of prime numbers, and define $I:=\sum_{p\in\mathbb{P}}\inv{p}\insZ$. Then, $\supp(I)=\{D_M\mid M\in\Max(D)\}$, so $A$ should be $\insZ$ itself; however, if $dI\subseteq D$ then $d$ should be divisible by every prime number, which cannot happen.
\end{enumerate}
\end{oss}

We want to use Proposition \ref{prop:skover} to decompose any semistar operation $\star$ into fractional star operations. We need another definition.
\begin{defin}
Let $D$ be a semilocal Pr\"ufer domain, and let $\SkOver(D)$ be the skeleton of $\Over(D)$. Let $\star\in\inssemistar(D)$. The \emph{support} of $\star$ is the set
\begin{equation*}
\supp(\star):=\{A\in\SkOver(D)\mid A^\star\in\insfracid(A)\}.
\end{equation*}
We denote the set of semistar operations on $D$ with support $\Delta$ as $\inssemisupp{\Delta}(D)$.
\end{defin}
Note that $\supp(\star)$ is always closed by intersections, since if $A^\star\in\insfracid(A)$ and $B^\star\in\insfracid(B)$ then $(A\cap B)^\star\subseteq A^\star\cap B^\star\in\insfracid(A\cap B)$. Moreover, the quotient field $K$ is always included in $\supp(\star)$.

An equivalent definition of $\supp(\star)$ is the set of elements $A$ of $\SkOver(D)$ such that $\star$ restricts to a fractional star operation on $A$. Hence, given any set $\Delta$ such that $\inssemisupp{\Delta}(D)\neq\emptyset$, we have a map
\begin{equation*}
\begin{aligned}
\gamma_\Delta:\inssemisupp{\Delta}(D) & \longrightarrow \prod\{\insfstar(A)\mid A\in\Delta\}\\
\star & \longmapsto (\star|_{\insfracid(A)})_{A\in\Delta}.
\end{aligned}
\end{equation*}

\begin{prop}\label{prop:gamma-inj}
Let $D$ be a semilocal Pr\"ufer domain, $\Delta\subseteq\SkOver(D)$, and let $\gamma_\Delta$ be defined as above. Then, $\gamma_\Delta$ is injective.
\end{prop}
\begin{proof}
Suppose $\gamma_\Delta(\star_1)=\gamma_\Delta(\star_2)=\gamma$, and let $I\in\inssubmod(D)$. By Proposition \ref{prop:skover}, $I\in\insfracid(A)$ for a unique $A\in\SkOver(D)$. If $A\in\Delta$, then $I^{\star_1}$ and $I^{\star_2}$ are equal to $I^{\gamma_A}$, where $\gamma_A$ is the component of $\gamma_\Delta$ with respect to $A$; hence $I^{\star_1}=I^{\star_2}$.

On the other hand, if $A\notin\Delta$, let $B$ be the smallest element of $\Delta$ containing $A$; it exists since $\Delta$ is closed by intersections. Then, $I^\star=(IA)^\star=(IA^\star)^\star=(IB)^\star$ for every $\star\in\inssemisupp{\Delta}(D)$; in particular, $I^\star=(IB)^{\gamma_\Delta(\star)_B}$. Since $\gamma_\Delta(\star_1)=\gamma_\Delta(\star_2)$, this again implies that $I^{\star_1}=I^{\star_2}$.

Therefore, $I^{\star_1}=I^{\star_2}$ in every case, and $\star_1=\star_2$.
\end{proof}

While $\gamma_\Delta$ is injective, it is usually very far from being surjective. For example, let $D$ be a one-dimensional Pr\"ufer domain with exactly two maximal ideals, $M$ and $N$; then, $\Theta=\{D_M,D_N\}$, and $\SkOver(D)=\{D,D_M,D_N,K\}=\Over(D)$. Suppose that $D_M$ is discrete while $D_N$ is not; then, by \cite[Chapter 31, Exercise 12]{gilmer} and \cite[Theorem 3.1]{twostar}, $\insfstar(D)=\insstar(D)$ is composed by two elements, the identity and the $v$-operation. Consider the element $(v_D,d_{D_M},d_{D_N},d_K)$ of $\insfstar(D)\times\insfstar(D_M)\times\insfstar(D_N)\times\insfstar(K)$, where $d_A$ indicates the identity on $A$ and $v_A$ the $v$-operation on $A$. Then, $N^{v_D}=D$, while $(ND_N)^{d_{D_N}}=ND_N$; in particular, $N^v\nsubseteq ND_N$, and thus $(v_D,d_{D_M},d_{D_N},d_K)$ cannot come from a semistar operation.

An inspection of this example shows that the problem lies in the fact that $v_D$ is ``not smaller'' than $d_{D_N}$; in terms of the $\gamma_\Delta$, we would like to impose the condition that $\gamma_\Delta(\star)|_A\leq\gamma_\Delta(\star)|_B$ whenever $A\subseteq B$. However, this condition doesn't really make sense as stated, since $\gamma_\Delta(\star)|_A$ and $\gamma_\Delta(\star)|_B$ live in different sets of closure operations. There are two possible approaches at this problems, both involving localizations of fractional star operations.

The first one uses localizations from one member of $\SkOver(D)$ to another. Indeed, if $A,B\in\SkOver(D)$ and $A\subseteq B$, then $B$ belongs to a Jaffard family of $A$ (explicitly, $\{B,T_1,\ldots,T_k\}$, where $T_1,\ldots,T_k$ are the elements of $\Theta$ that contain $A$ but not $B$). Hence, there is a localization map $\lambda_{A,B}:\insfstar(A)\longrightarrow\insfstar(B)$, and the condition becomes
\begin{equation*}
\lambda_{A,B}(\gamma_\Delta(\star)_A)\leq\gamma_\Delta(\star)_B.
\end{equation*}

The second approach, instead, uses localizations from $A$ to the members of the standard decomposition of $T$, and it is the one we will follow (mainly in view of the second part of Section \ref{sect:findim}).

Let $\Delta\subseteq\SkOver(D)$, and let $T\in\Theta$. The \emph{component of $\Delta$ with respect to $T$} is
\begin{equation*}
\Delta(T):=\{A\in\Delta\mid A\subseteq T\}.
\end{equation*}
Clearly, if $\Delta\neq\Lambda$ then there is a $T\in\Theta$ such that $\Delta(T)\neq\Lambda(T)$. A special case is $\Delta=\{K\}$: in this case, each $\Delta(T)$ is empty, and $\inssemisupp{\Delta}(D)=\{\wedge_{\{K\}}\}$.

Let now $A\in\Delta(T)$. Since $T$ belongs to a Jaffard family of $A$, there is a localization map $\lambda_{A,T}:\insfstar(A)\longrightarrow\insfstar(T)$. Therefore, for every $\star\in\inssemisupp{\Delta}(D)$ we get a map
\begin{equation*}
\begin{aligned}
\Gamma_T(\star)\colon\Delta(T) & \longrightarrow\insfstar(T) \\
A & \longmapsto \lambda_{A,T}(\star|_{\insfracid(A)}).
\end{aligned}
\end{equation*}

\begin{prop}\label{prop:GammaTstar}
Let $D,\Theta,\Delta$ as above; let $T\in\Theta$ and $\star\in\inssemisupp{\Delta}(D)$, and define $\Gamma_T(\star)$ as above. Then, $\Gamma_T(\star)$ is order-preserving.
\end{prop}
\begin{proof}
Let $A,B\in\Delta(T)$, $A\subseteq B$, and take any $\star\in\inssemisupp{\Delta}(D)$. Let $I$ be any integral ideal of $T$, and let $J:=I\cap A$; then, $JT=I$, and also $JBT=I$. Hence, by definition,
\begin{equation*}
I^{\lambda_{A,T}(\star|_{\insfracid(A)})}=J^\star A\subseteq (JB)^\star A=I^{\lambda_{B,T}(\star|_{\insfracid(B)})}.
\end{equation*}
Thus, $\lambda_{A,T}(\star|_{\insfracid(A)})\leq\lambda_{B,T}(\star|_{\insfracid(B)})$, as requested.
\end{proof}

If $Q_1$ and $Q_2$ are partially ordered sets, we denote by $\hom(Q_1,Q_2)$ the set of order-preserving maps between $Q_1$ and $Q_2$. This set is partially ordered; if $\phi,\psi\in\hom(Q_1,Q_2)$, then $\phi\leq\psi$ is $\phi(x)\leq\psi(x)$ for every $x\in Q_1$.

\begin{teor}\label{teor:sstar->hom}
Let $D$ be a semilocal Pr\"ufer domain with quotient field $K$, and let $\Theta$ be its standard decomposition; let $\Delta\neq\{K\}$ be a subset of $\SkOver(D)$ containing $K$ that is closed by intersections. The map
\begin{equation*}
\begin{aligned}
\Gamma_\Delta\colon\inssemisupp{\Delta}(D) & \longrightarrow\prod\{\hom(\Delta(T),\insfstar(T))\mid T\in\Theta,\Delta(T)\neq\emptyset\} \\
\star & \longmapsto (\Gamma_T(\star))_{T\in\Theta}
\end{aligned}
\end{equation*}
is an order isomorphism.
\end{teor}
\begin{proof}
By Proposition \ref{prop:GammaTstar}, $\Gamma:=\Gamma_\Delta$ is well-defined and order-preserving. To show that it is an isomorphism, we define an inverse.

For every $T\in\Theta$ such that $\Delta(T)\neq\emptyset$, let $\varphi_T\in\hom(\Delta(T),\insfstar(T))$. Take an $I\in\inssubmod(D)$; by Proposition \ref{prop:skover}, there is an $A\in\SkOver(D)$ such that $I\in\insfracid(A)$, and there is a $B\in\Delta$ such that $A^\star\in\insfracid(B)$. Then, we define
\begin{equation*}
I^{\star}:=\bigcap_{\substack{T\in\Theta\\ \Delta(T)\neq\emptyset}}(IBT)^{\varphi_T(B)}=\bigcap_{\substack{T\in\Theta\\ T\supseteq B}}(IT)^{\varphi_T(B)}.
\end{equation*}
We first claim that the map $\star$ so defined is a semistar operation.

Clearly, $\star$ is extensive and $(xI)^\star=x\cdot I^\star$ for every $x$ and every $I$ (since $I\in\insfracid(A)$ implies $xI\in\insfracid(A)$). To see that it is order-preserving, let $I,J\in\inssubmod(D)$, $I\subseteq J$. If $I,J\in\insfracid(A)$ for some $A\in\SkOver(D)$ the claim is trivial. If $I\in\insfracid(A)$ and $J\in\insfracid(A')$, then $A\subseteq A'$; if $A^\star\in\insfracid(B)$ and $A'^\star\in\insfracid(B')$, then also $B\subseteq B'$, and thus $IBT\subseteq JB'T$. Since $\varphi_T$ is order-preserving, we have $(IBT)^{\varphi_T(B)}\subseteq(JB'T)^{\varphi_T(B')}$; since this happens for all $T$, we have $I^\star\subseteq J^\star$, and $\star$ is order-preserving.

We need to show that $\star$ is idempotent. We note that, if $T\supseteq B$, then $IT\neq K$; therefore, by the proof of Proposition \ref{prop:skover}, $I^\star$ is a fractional ideal over $B$. Thus,
\begin{equation*}
(I^\star)^\star=\bigcap_{\substack{T\in\Theta\\ T\supseteq B}}\left[\left(\bigcap_{\substack{U\in\Theta\\ U\supseteq B}}(IU)^{\varphi_U(B)}\right)\cdot T\right]^{\varphi_T(B)}=\bigcap_{\substack{T\in\Theta\\ T\supseteq B}}\left[\bigcap_{\substack{U\in\Theta\\ U\supseteq B}}(IU)^{\varphi_U(B)}T\right]^{\varphi_T(B)},
\end{equation*}
with the last equality holding since the innermost intersection is finite and each $T\in\Theta$ is flat. Each $(IU)^{\varphi_U(B)}$ is a $U$-module; thus, if $U\neq T$, then $(IU)^{\varphi_U(B)}T=K$. Hence, the calculation above reduces to
\begin{equation*}
\bigcap_{\substack{T\in\Theta\\ T\supseteq B}}\left[(IT)^{\varphi_T(B)}\right]^{\varphi_T(B)}=\bigcap_{\substack{T\in\Theta\\ T\supseteq B}}(IT)^{\varphi_T(B)}=I^\star
\end{equation*}
since each $\varphi_T(B)$ is idempotent. Hence, $\star$ is idempotent, and thus a semistar operation. Also, a direct computation shows that the support of $\star$ is exactly $\Delta$.

Therefore, we have a map
\begin{equation*}
\begin{aligned}
\Phi:=\Phi_\Delta\colon\prod\{\hom(\Delta(T),\insfstar(T))\mid T\in\Theta,\Delta(T)\neq\emptyset\} & \longrightarrow \inssemisupp{\Delta}(D)
\end{aligned}
\end{equation*}
sending $(\varphi_T)_{T\in\Theta}$ to the map $\star$ defined as above.

We need to show that $\Phi\circ\Gamma$ and $\Gamma\circ\Phi$ are the identity (on $\inssemisupp{\Delta}(D)$ and the product, respectively).

Let $\star\in\inssemisupp{\Delta}(D)$. Then, if $I\in\insfracid(A)$ and $A^\star\in\insfracid(B)$, the map $\Phi\circ\Gamma(\star)$ sends $I$ to
\begin{equation*}
\bigcap_{\substack{T\in\Theta\\ T\supseteq B}}(IT)^{\Gamma_T(\star)(B)}= \bigcap_{\substack{T\in\Theta\\ T\supseteq B}}(IT)^{\lambda_{B,T}(\star|_{\insfracid(B)})}= \bigcap_{\substack{T\in\Theta\\ T\supseteq B}}(IB)^\star T=(IB)^\star=I^\star,
\end{equation*}
with the second to last equality coming from the fact that $\{T\in\Theta\mid T\supseteq B\}$ is a Jaffard family on $B$; hence, $\Phi\circ\Gamma(\star)=\star$.

On the other hand, let $\boldsymbol{\varphi}=(\varphi_T)_{T\in\Theta}$ be an element of the product, and fix a $U\in\Theta$. The component with respect to $U$ of $\Gamma\circ\Phi(\boldsymbol{\varphi})$ sends a $B\in\Delta(U)$ to $\lambda_{B,U}(\Phi(\boldsymbol{\varphi})|_{\insfracid(B)})$. Let $I=JU$ be a fractional ideal of $U$, where $J$ is a fractional ideal of $D$; by definition, this map sends $I$ to
\begin{equation*}
J^{\Phi(\boldsymbol{\varphi})}U=\left[\bigcap_{\substack{T\in\Theta\\ T\supseteq B}}(JT)^{\varphi_T(B)}\right]U=\bigcap_{\substack{T\in\Theta\\ T\supseteq B}}(JT)^{\varphi_T(B)}U=(JU)^{\varphi_U(B)},
\end{equation*}
again by flatness, the finiteness of the intersection and the equality $TU=K$ for $T\neq U$. Hence, $\Gamma\circ\Phi(\boldsymbol{\varphi})$ acts on $\insfracid(B)$ as $\boldsymbol{\varphi}$. Since this happens for each $B$, we have $\Gamma\circ\Phi(\boldsymbol{\varphi})=\boldsymbol{\varphi}$.

Therefore, $\Gamma_\Delta$ and $\Phi_\Delta$ are inverses one of each other, and the theorem is proved.
\end{proof}

\begin{cor}
Let $D$ be a semilocal Pr\"ufer domain with quotient field $K$, and let $\Delta\subseteq\SkOver(D)$. Then, $\Delta=\supp(\star)$ for some $\star\in\inssemistar(D)$ if and only if $K\in\Delta$ and $\Delta$ is closed by intersections.
\end{cor}
\begin{proof}
The conditions are clearly necessary. If $\Delta=\{K\}$, then $\Delta=\supp\{\wedge_{\{K\}}\}$; if $\Delta\neq\{K\}$, by the previous theorem $\inssemisupp{\Delta}(D)$ is isomorphic to a product of nonempty sets, and thus it is nonempty.
\end{proof}

By definition, $\inssemistar(D)$ is the disjoint union of $\inssemisupp{\Delta}(D)$, as $\Delta$ ranges among the subsets of $\SkOver(D)$; or, equivalently, among those subsets that are closed by intersections. Therefore, in light of Theorem \ref{teor:sstar->hom}, we can view $\inssemistar(D)$ as the union of products of sets of order-preserving maps. To fully reconstruct the set of semistar operations from this union, we need also to consider the order structure.
\begin{prop}\label{prop:ordine-hom}
Let $D$ be a semilocal Pr\"ufer domain, let $\Theta$ be its standard decomposition, and let $\star_1,\star_2\in\inssemistar(D)$. Then, $\star_1\leq\star_2$ if and only if
\begin{enumerate}
\item $\supp(\star_1)\supseteq\supp(\star_2)$; and
\item for any $A\in\supp(\star_2)$ and every $T\in\Theta$ such that $T\supseteq A$, we have $\Gamma_T(\star_1)(A)\leq\Gamma_T(\star_2)(A)$.
\end{enumerate}
\end{prop}
\begin{proof}
Suppose first that $\star_1\leq\star_2$. If $A\in\supp(\star_2)$, then $A^{\star_1}\subseteq A^{\star_2}$, and thus $A^{\star_1}$ is a fractional ideal of $A$; hence, $A\in\supp(\star_1)$ and $\supp(\star_1)\supseteq\supp(\star_2)$. Moreover, $\star_1|_{\insfracid(A)}\leq\star_2|_{\insfracid(A)}$; since the localization to $T$ preserves the order, $\Gamma_T(\star_1)\leq\Gamma_T(\star_2)$.

Conversely, suppose that the two conditions hold. If $\supp(\star_2)=\{K\}$, then $\star_2=\wedge_\{K\}$ and the claim holds; suppose $\supp(\star_2)\neq\{K\}$, so that in particular $\supp(\star_2)(T)\neq\emptyset$ for some $T\in\Theta$. Let $I$ be a $D$-submodule of the quotient field $K$; then, $I\in\insfracid(B)$ for some $B\in\SkOver(D)$. Let $A_i$ be the element of $\SkOver(D)$ such that $B^{\star_i}$ is a fractional ideal over $A_i$; since $\supp(\star_1)\supseteq\supp(\star_2)$, we have $A_1\subseteq A_2$. Then, 
\begin{equation*}
I^{\star_1}=(IA_1)^{\star_1}\subseteq (IA_2)^{\star_1}
\end{equation*}
and $(IA_2)^{\star_2}=I^{\star_2}$, so we need only to show that $(IA_2)^{\star_1}\subseteq(IA_2)^{\star_2}$; equivalently, we can suppose that $A_2=B\in\supp(\star_2)$.

Since, by the proof of Theorem \ref{teor:sstar->hom}, the inverse of $\Gamma$ is $\Phi$, we have
\begin{equation*}
I^{\star_1}=\bigcap_{\substack{T\in\Theta\\ T\supseteq B}}(IT)^{\Gamma_T(\star_1)(B)}\subseteq\bigcap_{\substack{T\in\Theta\\ T\supseteq B}}(IT)^{\Gamma_T(\star_2)(B)}=I^{\star_2},
\end{equation*}
since by hypothesis $\Gamma_T(\star_1)(A)\leq\Gamma_T(\star_2)(A)$ for every $T$. Hence, $\star_1\leq\star_2$, as requested.
\end{proof}

\section{Pr\"ufer domains with the same semistar operations}
Theorem \ref{teor:sstar->hom} and Proposition \ref{prop:ordine-hom}, taken together, show that the structure of $\inssemistar(D)$ (both as a set and as a partially-ordered set) depends exclusively from the sets $\hom(\Delta(T),\insfstar(T))$; or rather, exclusively from the $\insfstar(T)$.

More precisely, let $D_1$ and $D_2$ be two semilocal Pr\"ufer domains, and let $\Theta_1$ and $\Theta_2$ be their standard decompositions. As it was observed after Definition \ref{def:skover}, if $\Theta_1$ and $\Theta_2$ have the same cardinality then the structure of $\SkOver(D_1)$ and $\SkOver(D_2)$ is the same; that is, there is an order isomorphism $\nu:\SkOver(D_1)\longrightarrow\SkOver(D_2)$. Moreover, a subset $\Delta\subseteq\SkOver(D_1)$ is closed by intersections if and only if so is $\nu(\Delta)$, since the intersection of the elements of $\Delta$ is exactly its infimum in the natural order of $\SkOver(D_1)$ (that is, the inclusion). In particular, the subsets of $D_1$ that can be a support of a $\star\in\inssemistar(D_1)$ correspond bijectively to the subsets of $D_2$ that can support a semistar operation on $D_2$. Besides, $\nu$ restricts to a bijection (which, for simplicity, we still call $\nu$) between $\Delta(T)$ and $\nu(\Delta)(\nu(T))$.

Suppose now that, besides $\nu$, we have an order-preserving map $\nu_T:\insfstar(T)\longrightarrow\insfstar(\nu(T))$, for some $T\in\Theta_1$. Then, for every $\Delta\subseteq\SkOver(D_1)$ (not containing only the quotient field $K_1$) closed by intersections, we have a map
\begin{equation*}
\begin{aligned}
\widehat{\nu_T}\colon\hom(\Delta(T),\insfstar(T)) & \longrightarrow \hom(\nu(\Delta)(\nu(T)),\insfstar(\nu(T)))\\
\psi & \longmapsto \nu_T\circ\psi\circ\nu^{-1},
\end{aligned}
\end{equation*}
which is bijective as soon as $\nu_T$ is bijective. Hence, if we are given a bijection $\nu_T$ for every $T\in\Theta$, for every $\Delta$ we can build a map
\begin{equation*}
\begin{aligned}
\boldnu\colon\prod_{\substack{T\in\Theta_1\\ \Delta(T)\neq\emptyset}}\hom(\Delta(T),\insfstar(T)) & \longrightarrow \prod_{\substack{U\in\Theta_2\\ \Delta(U)\neq\emptyset}}\hom(\nu(\Delta)(U),\insfstar(U))\\
(\varphi_T) & \longmapsto (\widehat{\nu_T}(\varphi_T)).
\end{aligned}
\end{equation*}
By composing $\boldnu$ with the bijections $\Gamma_\Delta$ and $\Gamma_{\nu(\Delta)}$, we therefore obtain a bijective and order-preserving map $\inssemisupp{\Delta}(D_1)\longrightarrow\inssemisupp{\nu(\Delta)}(D_2)$. Since also $\inssemisupp{\{K_1\}}(D_1)=\{\wedge_{\{K_1\}}\}$ is isomorphic to $\inssemisupp{\{K_2\}}(D_2)=\{\wedge_{\{K_2\}}\}$, we can join all the supports to obtain a bijection $\inssemistar(D_1)\longrightarrow\inssemistar(D_2)$, which (by Proposition \ref{prop:ordine-hom}) respects the order. We have proved the following.
\begin{prop}\label{prop:nuT}
Let $D_1$ and $D_2$ be two semilocal Pr\"ufer domains, and let $\Theta_1$ and $\Theta_2$ be their standard decompositions. If there is a bijection $\nu:\Theta_1\longrightarrow\Theta_2$ and, for every $T\in\Theta_1$, an order isomorphism $\nu_T:\insfstar(T)\longrightarrow\insfstar(\nu(T))$, then $\inssemistar(D_1)$ and $\inssemistar(D_2)$ are order isomorphic.
\end{prop}

Obviously, the problem with this result is that it is difficult to check the hypothesis that $\insfstar(T)$ and $\insfstar(\nu(T))$ are isomorphic; in particular, if the standard decomposition of $D_1$ is exactly $\{D_1\}$ (and so $\Theta_2=\{D_2\}$) the theorem is essentially a vacuous statement. To get a better version, we need to consider the structure of the spectrum.

Let $D$ be a Pr\"ufer domain. It is well-known that its spectrum $\Spec(D)$ is a rooted tree, with root $(0)$; in particular, we can construct the underlying homeomorphically irreducible tree associated to $\Spec(D)$ (see Section \ref{sect:Omef-irred}), which we denote by $\hiSpec(D)$. In particular, $(0)$ and the maximal ideals of $D$ belong to $\hiSpec(D)$.

If $D$ is semilocal, then $\hiSpec(D)$ is finite: indeed, if $V(P)\cap\Max(D)=V(Q)\cap\Max(D)$, then at least one between $P$ and $Q$ is not in $\hiSpec(D)$. Therefore, for any $P\in\hiSpec(D)$, $P\neq(0)$, there is a $Q\in\hiSpec(D)$ such that $Q\subsetneq P$ and no element of $\hiSpec(D)$ lies between $Q$ and $P$; i.e., $Q$ is directly below $P$ in $\hiSpec(D)$. We denote by $Z(P)$ the ring $D_P/QD_P\simeq(D/Q)_{P/Q}$; when $P=(0)$, we set $Z(P)$ as the quotient field of $D$. Clearly, $Z(P)$ is a valuation domain.

\begin{prop}\label{prop:nuSkOver}
Let $D_1,D_2$ be semilocal Pr\"ufer domains, and let $\Theta_1,\Theta_2$ be, respectively, the standard decompositions of $D_1$ and $D_2$. Suppose there is an order isomorphism $\nu:\hiSpec(D_1)\longrightarrow\hiSpec(D_2)$. Then, there is an order isomorphism $\overline{\nu}:\SkOver(D_1)\longrightarrow\SkOver(D_2)$ such that:
\begin{enumerate}
\item $\overline{\nu}$ restricts to a bijection from $\Theta_1$ to $\Theta_2$;
\item for every $P\in\hiSpec(D_1)$ and every $T\in\Theta_1$, $PT=T$ if and only if $\nu(P)\overline{\nu}(T)=\overline{\nu}(T)$.
\end{enumerate}
\end{prop}
\begin{proof}
Let $D$ be a Pr\"ufer domain. By \cite[Proposition 6.2]{starloc}, the elements of $\Theta$ are in bijective correspondence with the equivalence classes of the dependence relation on $\Max(D)$. Moreover, if $D$ is semilocal, for every equivalence class $\Delta$, there is a $P\in\Spec(D)$ such that $T=\bigcap\{D_M\mid P\subseteq M\}$; in particular, if $P$ is maximal with respect to this property, $P\in\hiSpec(D)$ and, in fact, $P$ is a minimal element of $\hiSpec(D)\setminus\{(0)\}$.

Thus, coming back to the notation of the statement, the map
\begin{equation*}
\begin{aligned}
\overline{\nu}_0\colon\Theta_1 & \longrightarrow\Theta_2 \\
\bigcap_{\substack{M\in\Max(D_1)\\ P\subseteq M}}(D_1)_M & \longmapsto \bigcap_{\substack{N\in\Max(D_2)\\ \nu(P)\subseteq N}}(D_2)_N= \bigcap_{\substack{M\in\Max(D_1)\\ P\subseteq M}}(D_2)_{\nu(M)}
\end{aligned}
\end{equation*}
is a well-defined bijection; we can subsequently extend it to the whole $\SkOver(D)$ by putting $\nu(T_1\cap\cdots\cap T_n)=\nu(T_1)\cap\cdots\cap\nu(T_n)$ for every $T_1,\ldots,T_n\in\Theta_1$, obtaining again a bijection.

The last point is a direct consequence of the construction.
\end{proof}

With this notation, we can state one of the main theorems of the paper.
\begin{teor}\label{teor:semistar}
Let $D_1,D_2$ be semilocal Pr\"ufer domains, and suppose that there is an order isomorphism $\nu:\hiSpec(D_1)\longrightarrow\hiSpec(D_2)$ such that, for every $P\in\hiSpec(D_1)$, there is an order isomorphism $\nu_P:\insfstar(Z(P))\longrightarrow\insfstar(Z(\nu(P)))$. Then, there are order isomorphisms
\begin{equation*}
\boldnu:\inssemistar(D_1)\longrightarrow\inssemistar(D_2)
\end{equation*}
and
\begin{equation*}
\boldnu_F:\insfstar(D_1)\longrightarrow\insfstar(D_2)
\end{equation*}
such that, for every $\Delta\subseteq\SkOver(D_1)$,
\begin{equation*}
\boldnu(\inssemisupp{\Delta}(D_1))=\inssemisupp{\overline{\nu}(\Delta)}(D_2),
\end{equation*}
where $\overline{\nu}:\SkOver(D_1)\longrightarrow\SkOver(D_2)$ is the bijection found in Proposition \ref{prop:nuSkOver}.
\end{teor}
\begin{proof}
We proceed by induction on the cardinality of $\hiSpec(D)$. For every $k\inN$, $k>0$, let:
\begin{description}
\item[$(SS_k)$] $\boldnu$ exists whenever the hypotheses hold and $|\hiSpec(D_1)|\leq n$;
\item[$(FS_k)$] $\boldnu_F$ exists whenever the hypotheses hold and $|\hiSpec(D_1)|\leq n$.
\end{description}
(Note that the existence of $\nu$ guarantees that $|\hiSpec(D_1)|=|\hiSpec(D_2)|$.) We will show that $(FS_2)$ is true and that $(FS_n)\Longrightarrow(SS_n)\Longrightarrow(FS_{n+1})$; by induction, this will prove $(FS_n)$ and $(SS_n)$ for every $n$. Note that $(FS_1)$ and $(SS_1)$ are trivial, since they correspond to the case where $D_1$ and $D_2$ are fields.

\medskip

$(FS_2)$. If $|\hiSpec(D_1)|=2$, then $D_1$ and $D_2$ are valuation domains; hence, $\hiSpec(D_1)=\{(0),M\}$ (where $M$ is the maximal ideal of $D_1$) and $Z(M)=D_1$. Hence, the claim is just the hypothesis $\insfstar(Z(P))\leftrightarrow\insfstar(Z(\nu(P)))$. 

\medskip

$(FS_n)\Longrightarrow(SS_n)$ can be proved by following the reasoning of the proof of Proposition \ref{prop:nuT}, since if $T$ is in the standard decomposition of $D$ then $|\hiSpec(T)|\leq|\hiSpec(D)|$.

\medskip

$(SS_n)\Longrightarrow(FS_{n+1})$. Suppose first that $\Theta_1$ is a singleton, i.e., that $\Theta_1=\{D_1\}$. Then, there is a $P\in\hiSpec(D_1)$ contained in every maximal ideal of $D_1$, and every overring of $D_1$ (except for the quotient field $K_1$), is a fractional ideal of $D_1$: therefore, $\insfstar(D_1)=\inssemistar(D_1)\setminus\{\wedge_{\{K_1\}}\}$. By Proposition \ref{prop:cutting}, $\insfstar(D_1)$ is order-isomorphic to the ordinal union of $\inssemistar(D_1/P)$ and $\inssemistar((D_1)_P)\setminus\{d,\wedge_{\{K_1\}}\}$, and analogously $\insfstar(D_2)\simeq\inssemistar(D_2/\nu(P))\oplus(\inssemistar((D_2)_{\nu(P)})\setminus\{d,\wedge_{\{K_2\}}\})$.

We have $|\hiSpec(D_1/P)|=|\hiSpec(D)|-1$ and $|\hiSpec((D_1)_P)|=2$; by inductive hypothesis, and since the hypotheses of the theorem descend to these cases, we have order isomorphisms $\inssemistar(D_1/P)\simeq\inssemistar(D_2/\nu(P))$, while $\inssemistar((D_1)_P)\simeq\inssemistar((D_2)_{\nu(P)})$. Hence, there is an order isomorphism $\boldnu:\insfstar(D_1)\longrightarrow\insfstar(D_2)$.

Suppose now that $\Theta_1$ is not a singleton. By Proposition \ref{prop:cutting}, there is an order isomorphism between $\insfstar(D_1)$ and $\prod\{\insfstar(T)\mid T\in\Theta\}$, and analogously for $D_2$; moreover, as in the previous case, $\insfstar(T)=\inssemistar(T)\setminus\{\wedge_{\{K_1\}}\}$. Since $\Theta_1$ is not a singleton,  $|\hiSpec(T)|<|\hiSpec(D_1)|$ for every $T\in\Theta$; applying the inductive hypothesis, we have order isomorphisms $\nu_T:\inssemistar(T)\longrightarrow\inssemistar(\overline{\nu}(T))$, which (by the previous part of the proof) descend to order isomorphisms $\nu'_T:\insfstar(T)\longrightarrow\insfstar(\overline{\nu}(T))$. Therefore, we get an order isomorphism $\boldnu_F:\insfstar(D_1)\longrightarrow\insfstar(D_2)$ just by taking the product of the $\nu'_T$.

By induction, the claim is proved.
\end{proof}

\section{Star and (semi)star operations}
Theorem \ref{teor:semistar} shows that the sets $\inssemistar(D)$ and $\insfstar(D)$ of (respectively) the semistar operations and the fractional star operations on $D$ depends exclusively on $\hiSpec(D)$ and the semistar operations on the rings $Z(P)$. However, these properties are not enough to determine which operations close $D$, i.e., which closures are star or (semi)star operations.

For example, let $(V,M_V)$ be a one-dimensional valuation domain with $M_V$ not principal, and let $(W,M_W)$ be a two-dimensional valuation domain such that $M_W$ is principal, as well as $PW_P$ (where $P$ is the other nonzero prime of $W$). Then, $\hiSpec(V)=\{0,M_V\}$ correspond bijectively to $\hiSpec(W)=\{0,M_W\}$; moreover, both $\insfstar(V)$ and $\insfstar(W)$ are linearly ordered sets with three elements, so that they are order-isomorphic. However, there are two semistar operations closing $V$ (the identity and the $v$-operation) while only one closing $W$ (the identity). Hence, the bijection $\boldnu:\inssemistar(V)\longrightarrow\inssemistar(W)$ given by Theorem \ref{teor:semistar} does not restrict to a bijection $\boldnu:\inssmstar(V)\longrightarrow\inssmstar(W)$. In this section, we determine which hypothesis we have to add to obtain an analogous result.

We start with characterizing (semi)star operations through the map $\Gamma$.
\begin{prop}\label{prop:(semi)Gamma}
Let $D$ be a semilocal Pr\"ufer domain, $\Theta$ its standard decomposition, $\star\in\inssemistar(D)$; for each $T\in\Theta$, let $\Gamma_T(\star)$ be the map defined before Proposition \ref{prop:GammaTstar}. Then, $D=D^\star$ if and only if $D\in\supp(\star)$ and $\Gamma_T(\star)(D)\in\insstar(T)$ for every $T\in\Theta$.
\end{prop}
\begin{proof}
If $D=D^\star$, then $D\in\supp(\star)$ (since $D$ is always in $\SkOver(D)$), and thus $D\in\Delta(T)$ for every $T\in\Theta$. By definition, $\Gamma_T(\star)(D)=\lambda_{D,T}(\star|_{\insfracid(D)})$; however, $D=D^{\star|_{\insfracid(D)}}$, and thus
\begin{equation*}
T^{\Gamma_T(\star)(D)}=(DT)^{\Gamma_T(\star)(D)}=D^\star T=T,
\end{equation*}
and $\Gamma_T(\star)(D)\in\insstar(T)$.

Conversely, suppose the two properties hold, and let $\Delta:=\supp(\star)$. By the proof of Theorem \ref{teor:sstar->hom}, we have
\begin{equation*}
D^\star=D^{\Phi_\Delta\circ\Gamma_\Delta(\star)}=\bigcap_{T\in\Theta}(DT)^{\Gamma_T(\star)(D)},
\end{equation*}
noting that each $\Delta(T)$ contains $D$ and thus it is nonempty. By hypothesis, each $\Gamma_T(\star)(D)$ closes $T$; thus, $D^\star=\bigcap_{T\in\Theta}T=D$. The claim is proved.
\end{proof}

If $D\in\Delta(T)$, let us thus denote by $\tildhom(\Delta(T),\insfstar(T))$ the set of order-preserving maps $\psi$ from $\Delta(T)$ to $\insfstar(T)$ such that $\psi(D)\in\insstar(D)$. The previous proposition can thus be rewritten as follows: given a $\Delta\subseteq\SkOver(D)$ closed by intersections and containing $D$ and $K$, there is a bijection between $\inssmsupp{\Delta}(D)$ (i.e., the set of (semi)star operations with support $\Delta$) and the product $\prod\{\tildhom(\Delta(T),\insfstar(T))\mid T\in\Theta\}$.

We thus obtain immediately an analogue of Proposition \ref{prop:nuT}: if $D_1,D_2$ are semilocal Pr\"ufer domains, with standard decompositions $\Theta_1,\Theta_2$, and there are bijections $\nu:\Theta_1\longrightarrow\Theta_2$ and $\nu_T:\insfstar(T)\longrightarrow\insfstar(\nu(T))$, for every $T\in\Theta$, and if $\nu_T(\insstar(T))=\insstar(\nu(T))$, then the order isomorphism $\boldnu:\inssemistar(D_1)\longrightarrow\inssemistar(D_2)$ restricts to a bijection from $\inssmstar(D_1)$ to $\inssmstar(D_2)$. We can actually say more.
\begin{teor}\label{teor:(semi)star}
Let $D_1,D_2$ be semilocal Pr\"ufer domains, and suppose that there is an order isomorphism $\nu:\hiSpec(D_1)\longrightarrow\hiSpec(D_2)$ such that:
\begin{enumerate}
\item for every $P\in\hiSpec(D_1)$, there is an order isomorphism $\nu_P:\insfstar(Z(P))\longrightarrow\insfstar(Z(\nu(P)))$;
\item for every $M\in\Max(D_1)$, $M(D_1)_M$ is principal if and only if $\nu(M)(D_2)_{\nu(M)}$ is principal.
\end{enumerate}
Then, the maps $\boldnu:\inssemistar(D_1)\longrightarrow\inssemistar(D_2)$ and $\boldnu_F:\insfstar(D_1)\longrightarrow\insfstar(D_1)$ found in Theorem \ref{teor:semistar} restrict to order isomorphisms $\boldnu_{(S)}:\inssmstar(D_1)\longrightarrow\inssmstar(D_2)$ and $\boldnu_S:\insstar(D_1)\longrightarrow\insstar(D_2)$.
\end{teor}
\begin{proof}

By Theorem \ref{teor:semistar}, the hypothesis guarantee that $\boldnu$ and $\boldnu_F$ are order isomorphisms.

The proof follows the same reasoning of the the proof of Theorem \ref{teor:semistar}: for every $k\inN$, $k>0$, let:
\begin{description}
\item[$(Ss_k)$] $\boldnu_{(S)}$ exists whenever the hypotheses hold and $|\hiSpec(D_1)|\leq n$;
\item[$(S_k)$] $\boldnu_S$ exists whenever the hypotheses hold and $|\hiSpec(D_1)|\leq n$.
\end{description}
Then, $(S_2)$ is true because, if $V$ is a valuation domain, $M$ is principal if and only if $|\insstar(D_1)|=1$, while $M$ is not principal if and only if $|\insstar(D_1)|=2$; furthermore, $(S_n)\Longrightarrow(Ss_n)$ follows from the reasoning before the statement of the theorem. 

To show $(Ss_n)\Longrightarrow(S_{n+1})$, we first suppose that $\Theta_1$ is a singleton: then, $\insstar(D_1)=\inssmstar(D_1)$, and the isomorphism between $\insfstar(D_1)$ and $\inssemistar(D_1/P)\oplus(\inssemistar((D_1)_P)\setminus\{d,\wedge_{\{K_1\}}\})$  (Proposition \ref{prop:cutting}) restricts to an isomorphism between $\insstar(D_1)$ and $\inssmstar(D_1/P)$ ; the inductive hypothesis shows that $\boldnu$ restricts to a bijection $\boldnu_{(S)}:\insstar(D_1)\longrightarrow\insstar(D_2)$.

On the other hand, if $\Theta_1$ is not a singleton, we use \cite[Theorem 5.4]{starloc} to reduce $\insstar(D_1)$ to the product $\prod\{\insstar(T)\mid T\in\Theta\}$, and then apply the inductive hypothesis on each $T$.

The claim then follows by induction.
\end{proof}

Suppose now that $D$ is a semilocal Pr\"ufer domain whose standard decomposition if $\{D\}$. As we have observed multiple times, there is a unique element of $\hiSpec(D)$ just above $(0)$: call it $P$. Then, $\insstar(D)$ corresponds to $\inssmstar(D/P)$; in particular, $\insstar(D)$ cannot depend on $\inssemistar(Z(P))$, since it depends exclusively on $D/P$.

We can thus get the following results.
\begin{teor}\label{teor:star-nomin}
Let $D_1,D_2$ be semilocal Pr\"ufer domains, and suppose that there is an order isomorphism $\nu:\hiSpec(D_1)\longrightarrow\hiSpec(D_2)$ such that:
\begin{enumerate}
\item for every $P\in\hiSpec(D_1)$ such that $P$ is not minimal in $\hiSpec(D_1)\setminus\{(0)\}$, there is an order isomorphism $\nu_P:\insfstar(Z(P))\longrightarrow\insfstar(Z(\nu(P)))$;
\item for every $M\in\Max(D_1)$, $M(D_1)_M$ is principal if and only if $\nu(M)(D_2)_{\nu(M)}$ is principal.
\end{enumerate}
Then, there is an order isomorphism $\boldnu_S$ between $\insstar(D_1)$ and $\insstar(D_2)$.
\end{teor}
\begin{proof}
By \cite[Theorem 5.4]{starloc}, $\insstar(D_1)\simeq\prod\{\insstar(T)\mid T\in\Theta_1\}$ and $\insstar(D_2)\simeq\prod\{\insstar(U)\mid U\in\Theta_2\}$ (where $\Theta_1$ and $\Theta_2$ are the standard decompositions of $D_1$ and $D_2$). By the previous reasoning, $\insstar(T)\simeq\inssmstar(T/P_T)$ (where $P_T$ is the minimal element of $\hiSpec(T)\setminus\{(0)\}$); we can apply Theorem \ref{teor:(semi)star} to each $T/P_T$, obtaining order isomorphisms $\boldnu_S(T):\insstar(T)\longrightarrow\insstar(\nu(T))$. To conclude, we just take $\boldnu_S$ to be the product of all the $\boldnu_S(T)$.
\end{proof}

Notice that, under the hypotheses of the last theorem, the isomorphisms $\boldnu$ and $\boldnu_F$ need not to exist, and thus Theorem \ref{teor:(semi)star} cannot be reduced to a corollary of Theorem \ref{teor:star-nomin}.

\section{The finite-dimensional case}\label{sect:findim}
The results in the previous two sections can be simplified if we work in the finite-dimensional case. Indeed, suppose $V$ is a finite-dimensional valuation domain: then, $V$ admits only a finite number of overrings (its localizations) and each one admits a finite number of (semi)star operations (at most two, the identity and the $v$-operation). Therefore, $\inssemistar(V)$ is finite; since it is also linearly ordered, it is actually characterized by its cardinality.

Following this idea, we introduce the functions
\begin{equation*}
\begin{aligned}
\omega\colon\hiSpec(D) & \longrightarrow \insN^+\\
P & \longmapsto |\insfstar(Z(P))|
\end{aligned}
\end{equation*}
and
\begin{equation*}
\begin{aligned}
\epsilon\colon\Spec(D) & \longrightarrow \{1,2\}\\
P & \longmapsto |\insstar(D_P)|.
\end{aligned}
\end{equation*}
We note that $\omega$ can also be thought of as a function from the set of the edges of $\hiSpec(D)$ to $\insN^+$: if $E$ is an edge from $Q$ to $P$, then $\omega(E)$ would be defined as $\omega(P)$. Note also that $\omega((0))$ is always equal to 1.

The following propositions establish the properties of $\omega$ and $\epsilon$ and their connection.
\begin{prop}\label{prop:omegaepsilon-V}
Let $V$ be a valuation domain with maximal ideal $M$.
\begin{enumerate}[(a)]
\item\label{prop:omegaepsilon-V:semi} $|\inssemistar(V)|=\omega(M)+1$.
\item\label{prop:omegaepsilon-V:(semi)} $|\inssmstar(V)|=\epsilon(M)$.
\item\label{prop:omegaepsilon-V:princ} $\epsilon(M)=1$ if and only if $M$ is principal.
\item\label{prop:omegaepsilon-V:somma} Let $\mathcal{I}$ be the set of nonzero idempotent prime ideals of $V$ and $\mathcal{N}$ be the set of nonzero nonidempotent prime ideals of $V$. Then,
\begin{equation}\label{eq:omegaepsilon}
\omega(M)=\sum_{\substack{P\in\Spec(V)\\ P\neq(0)}}\epsilon(P)=|\mathcal{N}|+2\cdot|\mathcal{I}|.
\end{equation}
\end{enumerate}
\end{prop}
\begin{proof}
\ref{prop:omegaepsilon-V:semi} and \ref{prop:omegaepsilon-V:(semi)} follow from the fact that every overring of $V$ different from $K$ is both a localization of $V$ and a fractional ideal of $V$, and they also show the first equality of \eqref{eq:omegaepsilon}. \ref{prop:omegaepsilon-V:princ} is well known. The second equality of \eqref{eq:omegaepsilon} follows from the fact that $P$ is nonidempotent if and only if $PV_P$ is principal, i.e., if and only if $\epsilon(P)=1$. \ref{prop:omegaepsilon-V:somma} is proved.
\end{proof}

\begin{prop}\label{prop:omegaepsilon}
Let $D$ be a semilocal finite-dimensional Pr\"ufer domain, and let $P\in\hiSpec(D)\setminus\{0\}$; let $Q$ be the element of $\hiSpec(D)$ directly below $P$. Let $\Delta:=\{A\in\Spec(D)\mid Q\subsetneq A\subseteq P\}$, and let $\mathcal{I}$ be the set of idempotent prime ideals of $D$ and $\mathcal{N}$ the set of nonidempotent prime ideals of $D$. Then,
\begin{equation*}
\omega(P)=\sum_{A\in\Delta}\epsilon(A)=|\Delta\cap\mathcal{N}|+2\cdot|\Delta\cap\mathcal{I}|.
\end{equation*}
\end{prop}
\begin{proof}
The claim follows directly from Proposition \ref{prop:omegaepsilon-V} and the fact that a prime ideal $A$ such that $Q\subsetneq A\subseteq P$ is idempotent if and only if its extension in $Z(P)$ is.
\end{proof}

With this terminology, Theorem \ref{teor:semistar} translates immediately to the following statement.
\begin{teor}\label{teor:semistar-fd}
Let $D_1,D_2$ be semilocal Pr\"ufer domain of finite dimension. Suppose there is an order-preserving map $\nu:\hiSpec(D_1)\longrightarrow\hiSpec(D_2)$ such that $\omega(P)=\omega(\nu(P))$ for every $P\in\hiSpec(D_1)$. Then, there are order isomorphisms
\begin{equation*}
\boldnu:\inssemistar(D_1)\longrightarrow\inssemistar(D_2)
\end{equation*}
and
\begin{equation*}
\boldnu_F:\insfstar(D_1)\longrightarrow\insfstar(D_2)
\end{equation*}
such that, for every $\Delta\subseteq\SkOver(D_1)$ closed by intersections,
\begin{equation*}
\boldnu(\inssemisupp{\Delta}(D_1))=\inssemisupp{\overline{\nu}(\Delta)}(D_2),
\end{equation*}
where $\overline{\nu}:\SkOver(D_1)\longrightarrow\SkOver(D_2)$ is the bijection found in Proposition \ref{prop:nuSkOver}.
\end{teor}
\begin{proof}
Since $\insfstar(V)$ is linearly ordered for every valuation domain $V$, the condition $\omega(P)=\omega(\nu(P))$ implies that there is an isomorphism between $\insfstar(Z(P))$ and $\insfstar(Z(\nu(P)))$. Hence, we can apply Theorem \ref{teor:semistar}.
\end{proof}

In the same way, we have analogues of the results about (semi)star operations.
\begin{teor}\label{teor:(semi)star-fd}
Let $D_1,D_2$ be semilocal Pr\"ufer domains, and suppose that there is an order isomorphism $\nu:\hiSpec(D_1)\longrightarrow\hiSpec(D_2)$ such that $\epsilon(M)=\epsilon(\nu(M))$ for every $M\in\Max(D_1)$.
\begin{enumerate}[(a)]
\item\label{teor:(semi)star-fd:(semi)} If $\omega(P)=\omega(\nu(P))$ for $P\in\hiSpec(D_1)$, then the maps $\boldnu$ and $\boldnu_F$ found in Theorem \ref{teor:semistar} restrict to order isomorphisms $\boldnu_{(S)}:\inssmstar(D_1)\longrightarrow\inssmstar(D_2)$ and $\boldnu_S:\insstar(D_1)\longrightarrow\insstar(D_2)$.
\item\label{teor:(semi)star-fd:nomin} If $\omega(P)=\omega(\nu(P))$ for every $P\in\hiSpec(D)$ such that $P$ is not minimal in $\hiSpec(D_1)\setminus\{(0)\}$, then there is an order isomorphism $\boldnu_S$ between $\insstar(D_1)$ and $\insstar(D_2)$.
\end{enumerate}
\end{teor}
\begin{proof}
\ref{teor:(semi)star-fd:(semi)} follows from Theorem \ref{teor:(semi)star}, while \ref{teor:(semi)star-fd:nomin} follows from Theorem \ref{teor:star-nomin}.
\end{proof}

Let now $\mathcal{P}$ be a finite rooted tree which is also homeomorphically irreducible. Then, there are finite-dimensional semilocal Pr\"ufer domains such that $\hiSpec(D)\simeq\mathcal{P}$ \cite[Theorem 3.1]{lewis-spectrum}; by Theorem \ref{teor:semistar-fd}, the cardinality of $\inssemistar(D)$ depends only on $\omega(P)$, as $P$ ranges in $\hiSpec(D)$. Hence, if we label the elements of $\mathcal{P}$ as $\{(0),P_1,\ldots,P_k\}$, we can define a function $\Sigma_\mathcal{P}:\insN^k\longrightarrow\insN$ such that $\Sigma_\mathcal{P}(a_1,\ldots,a_k)$ is the cardinality of $\inssemistar(D)$, where $\hiSpec(D)\simeq\mathcal{P}$ and $\omega(P_i)=a_i$ for each $i$.

Similarly, if $\mathcal{P}\simeq\hiSpec(D)=\{(0),P_1,\ldots,P_k,M_1,\ldots,M_t\}$, where $M_1,\ldots,M_t$ are the maximal ideals of $D$, we define $\widetilde{\Sigma}_\mathcal{P}$ as the function $\insN^{k+t}\times\{1,2\}^t\longrightarrow\insN$ such that $\Sigma_\mathcal{P}(a_1,\ldots,a_k,b_1,\ldots,b_t,c_1,\ldots,c_t)$ is the cardinality of $\inssmstar(D)$, where $\omega(P_i)=a_i$, $\omega(M_j)=b_j$ and $\epsilon(M_l)=c_l$ for each $i,j,l$.

To study what kind of functions $\Sigma_\mathcal{P}$ and $\widetilde{\Sigma}_\mathcal{P}$ are, we shall use the following extension of \cite[Theorem 1]{stanley-posets}; we will denote by $\underline{n}$ the set $\{1,\ldots,n\}$, endowed with the usual ordering.
\begin{prop}\label{prop:poset}
Let $\mathcal{P},\mathcal{Q}$ be two partially ordered sets, and let $H_{\mathcal{P},\mathcal{Q}}(n):=|\hom(\mathcal{P},\mathcal{Q}\oplus\underline{n})|$. Then, $H_{\mathcal{P},\mathcal{Q}}$ is a polynomial of degree $|\mathcal{P}|$.
\end{prop}
\begin{proof}
For any order-preserving map $\psi:\mathcal{P}\longrightarrow\mathcal{Q}\oplus\underline{n}$, let ${\downarrow}\psi:=\{p\in\mathcal{P}\mid \psi(p)\in\mathcal{Q}\}$ and ${\uparrow}\psi:=\{p\in\mathcal{P}\mid \psi(p)\in\underline{n}\}$. Then, if $p\in{\downarrow}\psi$ and $q\in{\uparrow}\psi$, we have $p\leq q$. We can see any $\psi\in\hom(\mathcal{P},\mathcal{Q}\oplus\underline{n})$ as the union of a map $\psi_1:{\downarrow}\psi\longrightarrow\mathcal{Q}$ and a map $\psi_2:{\uparrow}\psi\longrightarrow\underline{n}$, both of which are order-preserving, that are independent one from the other.

For any $\Delta$, let $\hom^\Delta(\mathcal{P},\mathcal{Q}\oplus\underline{n}):=\{\psi\in\hom(\mathcal{P},\mathcal{Q}\oplus\underline{n})\mid {\downarrow}\psi=\Delta\}$. Clearly, $\hom(\mathcal{P},\mathcal{Q}\oplus\underline{n})$ is the union of the various $\hom^\Delta$; moreover, by the previous reasoning, if $\Delta={\downarrow}\psi$  for some $\psi$, we have
\begin{equation*}
|\hom^\Delta(\mathcal{P},\mathcal{Q}\oplus\underline{n})|=|\hom(\Delta,\mathcal{Q})|\cdot|\hom(\mathcal{P}\setminus\Delta,\underline{n})|.
\end{equation*}
For a fixed $\mathcal{Q}$, the first factor depends uniquely on $\Delta$. On the other hand, by \cite[Theorem 1]{stanley-posets}, the second factor is a polynomial $H_{\mathcal{P}\setminus\Delta}$ of degree $|\mathcal{P}\setminus\Delta|$. Since $H_{\mathcal{P},\mathcal{Q}}(n)$ is the sum of the cardinalities of the $\hom^\Delta$, also $H_{\mathcal{P},\mathcal{Q}}$ is a polynomial; moreover, there is an unique summand of maximal degree, namely $|\hom^\emptyset(\mathcal{P},\mathcal{Q}\oplus\underline{n})|=|\hom(\mathcal{P},\underline{n})|$, whose degree is $|\mathcal{P}|$. Hence, $H_{\mathcal{P},\mathcal{Q}}$ has degree $|\mathcal{P}|$.
\end{proof}

\begin{oss}~\label{oss:posets}
\begin{enumerate}
\item If $\mathcal{Q}=\emptyset$, the result above falls back to \cite[Theorem 1]{stanley-posets}.
\item\label{oss:posets:Qn} If $\mathcal{P}=\underline{k}$ is linearly ordered, we denote $H_{\underline{k},\emptyset}$ as $H_k$. Order-preserving maps from $\underline{k}$ to $\underline{n}$ correspond to ways of dividing $\underline{n}$ into $k$ (possibly empty) segments, or equivalently to combinations with repetition of $k$ elements in $\{1,\ldots,n\}$; therefore, $H_k=\binom{n+k-1}{k}$. For example, $H_1(n)=n$, while $H_2(n)=\frac{n(n+1)}{2}$ and $H_3(n)=\frac{n(n+1)(n+2)}{6}$.
\end{enumerate}
\end{oss}

\begin{teor}\label{teor:num-semistar}
Let $\mathcal{P}=\{0,p_1,\ldots,p_n\}$ be a finite rooted homeomorphically irreducible tree, with root $0$, and let $\{p_1,\ldots,p_k\}$ be the minimal elements of $\mathcal{P}\setminus\{0\}$. Then, for every $b_{k+1},\ldots,b_n\inN$, the function
\begin{equation*}
\pi_\mathcal{P}(X_1,\ldots,X_k):=\Sigma_\mathcal{P}(X_1,\ldots,X_k,b_{k+1},\ldots,b_n)
\end{equation*}
is a polynomial of degree $k\cdot 2^{k-1}$.
\end{teor}
\begin{proof}
Let $D$ be a semilocal finite-dimensional domain such that $\hiSpec(D)=\{(0),P_1,\ldots,P_n\}\simeq\mathcal{P}$, with $\omega(P_i)=b_i$ for $k<i\leq n$. By definition, the cardinality of $\inssemistar(D)$ is equal to the sum of the cardinalities of $\inssemisupp{\Delta}(D)$, as $\Delta$ ranges among the possible supports. Let $\Theta$ be the standard decomposition of $D$.

For every such $\Delta$, by Theorem \ref{teor:semistar} we have
\begin{equation*}
|\inssemisupp{\Delta}(D)|=\prod\{|\hom(\Delta(T),\insfstar(T))|:T\in\Theta,\Delta(T)\neq\emptyset\}|.
\end{equation*}
By Proposition \ref{prop:cutting}, $\insfstar(T)$ is equal to the union of $\inssemistar(T/P)$ and $\insfstar(T_P)\setminus\{d\}$, where $P$ is the minimal element of $\hiSpec(T)\setminus\{(0)\}$; moreover, $\inssemistar(T/P)$ has a maximum (namely $\wedge_{\{k\}}$, where $k$ is the quotient field of $T/P$), and thus we can write $\insfstar(T)$ as $\mathcal{Q}^{(T)}\oplus\underline{\omega(P)}$, where $\mathcal{Q}^{(T)}:=\inssemistar(T/P)\setminus\{\wedge_{\{k\}}\}$. Applying Proposition \ref{prop:poset}, we see that $|\hom(\Delta(T),\insfstar(T))|=H_{\Delta(T),\mathcal{Q}^{(T)}}(\omega(P))$ is a polynomial in $\omega(P)$ of degree $|\Delta(T)|$; hence, each $|\inssemisupp{\Delta}(D)|$ is a polynomial in $\omega(P_1),\ldots,\omega(P_k)$. In particular, $\pi_\mathcal{P}$ is a polynomial.

Moreover, the term of maximal degree of each $|\inssemisupp{\Delta}(D)|$ has degree $|\Delta(T)|$ in $\omega(P)$, where $P$ is the minimal element of $\hiSpec(T)\setminus\{(0)\}$; in particular, this degree is maximal when $\Delta(T)$ is just the set of intersections of the subsets of the standard decomposition $\Theta$ containing $T$, where it is $2^{k-1}$. Hence, the maximal term of $\pi_\mathcal{P}$ comes from the case $\Delta=\SkOver(D)$, where each $\omega(P)$ has degree $2^{k-1}$. It follows that the total degree of $\pi_\mathcal{P}$ is $k\cdot 2^{k-1}$.
\end{proof}

\begin{teor}\label{teor:num-(semi)star}
Let $\mathcal{P}:=\{0,p_1,\ldots,p_n,m_1,\ldots,m_t\}$ be a finite rooted homeomorphically irreducible tree, with root $0$, and let $\{p_1,\ldots,p_k\}$ be the minimal elements of $\mathcal{P}\setminus\{0\}$. Then, for every $b_{k+1},\ldots,b_n\inN$, $c_1,\ldots,c_t\in\{1,2\}$ the function
\begin{equation*}
\widetilde{\pi}_\mathcal{P}(X_1,\ldots,X_k):=\widetilde{\Sigma}_\mathcal{P}(X_1,\ldots,X_k,b_{k+1},\ldots,b_n,c_1,\ldots,c_t)
\end{equation*}
is a polynomial of degree $k(2^{k-1}-1)$.
\end{teor}
\begin{proof}
As in the proof of Theorem \ref{teor:num-semistar}, we need only to show that each $|\inssmsupp{\Delta}(D_P)|$ is a polynomial, and since we are considering (semi)star operations, we can consider only sets $\Delta$ containing $D$.

Consider a set $\Delta(T)$, and let $\Lambda(T)=\Delta(T)\setminus\{D\}$. For each $\star\in\insstar(T)$, set
\begin{equation*}
\tildhom_\star(\Delta(T),\insfstar(T)):=\{\psi\in\tildhom(\Delta(T),\insfstar(T))\mid \psi(D)=\star\}.
\end{equation*}
Then, the cardinality of $\tildhom_\star(\Delta(T),\insfstar(T))$ is equal to the cardinality of $\hom(\Lambda(T),\{\sharp\in\insfstar(T)\mid \sharp\geq\star\})$, which by Proposition \ref{prop:poset} is a polynomial of degree $|\Lambda(T)|=|\Delta(T)|-1$ in $\omega(P)$, where $P$ is the minimal element of $\hiSpec(T)\setminus\{(0)\}$ (note that a star operation on $T$ correspond to a star operation coming from $\inssemistar(T/P)$).

Following the reasoning of Theorem \ref{teor:num-semistar}, this is maximal when $|\Delta(T)|=2^{k-1}$; hence, $\widetilde{\pi}_\mathcal{P}$ is a polynomial of degree $k(2^{k-1}-1)$.
\end{proof}

A good measure of the complexity of the calculation of the polynomials $\pi_\mathcal{P}$ and $\widetilde{\pi}_\mathcal{P}$ is the \emph{height} $h(\mathcal{P})$ of $\mathcal{P}=\hiSpec(D)$, that is, the maximal length among the chains of $\mathcal{P}$. When the height is 0, $D$ is a field; hence, the first interesting case is when $h(\mathcal{P})=1$. In algebraic terms, this happens if and only if $D$ is $h$-local, that is, if $D$ is locally finite (which is automatic when $D$ is semilocal) and $D_MD_N=K$ for $M\neq K$ in $\Max(D)$ (see e.g. \cite{olberding_hlocal} for a study of Pr\"ufer $h$-local domains). 

In this case, the calculation of star and fractional star operations does not need the theory developed in this article; indeed, by \cite[Theorem 5.4]{starloc} (and Section \ref{sect:jaff}), if $D$ is $h$-local and $\Max(D)=\{M_1,\ldots,M_n\}$, then $|\insstar(D)|=\epsilon(M_1)\cdots\epsilon(M_n)$ while $|\insfstar(D)|=\omega(M_1)\cdots\omega(M_n)$. The case of semistar operations, on the other hand, is not so immediate, but it is a mere consequence of Theorem \ref{teor:num-semistar}.
\begin{cor}\label{cor:semistar-hloc}
There is a symmetric polynomial $\pi_n\in\insQ[X_1,\ldots,X_n]$ of degree $n\cdot 2^{n-1}$ such that, if $D$ is a $h$-local Pr\"ufer domain and $\Max(D)=\{M_1,\ldots,M_n\}$, then $|\inssemistar(D)|=\pi_n(\omega(M_1),\ldots,\omega(M_n))$.
\end{cor}
\begin{proof}
If $D$ is $h$-local, then $\hiSpec(D)=\{(0)\}\cup\Max(D)$. Then, $\pi_n$ is a polynomial by Theorem \ref{teor:num-semistar}, and it is obviously symmetric.
\end{proof}

The case of (semi)star operations is more interesting, since we can actually make the numbers $\epsilon(M_i)$ variables, instead of parameters as it was in Theorem \ref{teor:num-(semi)star}.
\begin{prop}\label{prop:(semi)star-dim1}
There is a polynomial $\widetilde{\pi}_n\in\insQ[X_1,\ldots,X_n,Y_1,\ldots,Y_n]$ of degree $n\cdot 2^{n-1}$ such that, if $D$ is a $h$-local Pr\"ufer domain and $\Max(D)=\{M_1,\ldots,M_n\}$, then $|\inssmstar(D)|=\widetilde{\pi}_n(\omega(M_1),\ldots,\omega(M_n),\epsilon(M_1),\ldots,\epsilon(M_n))$.
\end{prop}
\begin{proof}
As in the proof of Theorem \ref{teor:num-(semi)star}, we must calculate the cardinality of the sets $\tildhom_\star(\Delta(T),\insfstar(T)):=\{\psi\in\tildhom(\Delta(T),\insfstar(T))\mid \psi(D)=\star\}$, as $T$ ranges in the standard decomposition of $D$ and $\star\in\insstar(T)$.

Since $D$ is $h$-local, each $T$ is a localization at a maximal ideal of $D$; hence, each $T=D_P$ is a valuation domain, and the possible star operations $\star$ are the identity and the $v$-operation. If $\star$ is the identity $d$, then 
\begin{equation*}
|\tildhom_\star(\Delta(D_P),\insfstar(D_P))|=|\hom(\Lambda(D_P),\insfstar(D_P))|=H_{\Lambda(D_P),\emptyset}(\omega(P))
\end{equation*}
(where $\Lambda(D_P)=\Delta(D_P)\setminus\{D_P\}$). On the other hand, if $\star=v$, then 
\begin{equation*}
|\tildhom_\star(\Delta(D_P),\insfstar(D_P))|=|\hom(\Lambda(D_P),\insfstar(D_P)\setminus\{d\})|=H_{\Lambda(D_P),\emptyset}(\omega(P)-1).
\end{equation*}

The latter summand exists only when $\epsilon(P)=2$; therefore, we have
\begin{equation*}
|\tildhom(\Delta(D_P),\insfstar(D_P))|=H_{\Lambda(D_P),\emptyset}(\omega(P))+(\epsilon(P)-1)H_{\Lambda(D_P),\emptyset}(\omega(P)-1).
\end{equation*}
Putting all together, we see that $\widetilde{\pi}_n$ is a polynomial of degree $2^{n-1}-1$ in each $X_i$ and $1$ in each $Y_i$; the total degree is thus $n\cdot 2^{n-1}$.
\end{proof}

We can use these results, along with Proposition \ref{prop:cutting}, to study star and fractional star operations when the height of $\hiSpec(D)$ is 2.
\begin{prop}\label{prop:num-star-fstar}
Let $D$ be a semilocal Pr\"ufer domain, and let $\hiSpec(D)=\{(0)\}\sqcup\mathcal{A}\sqcup\Max(D)$; suppose that the elements of $\mathcal{A}$ are pairwise not comparable. For any $P\in\mathcal{A}$, let $M(P):=\{M\in\Max(D)\mid P\subseteq M\}=\{M_{P,1},\ldots,M_{P,|M(P)|}\}$. Let $\omega$, $\epsilon$, $\pi_n$ and $\widetilde{\pi}_n$ as above. Then,
\begin{equation*}
|\insfstar(D)|=\prod_{P\in\mathcal{A}}[\pi_{|M(P)|}(\omega(M_{P,1}),\ldots,\omega(M_{P,|M(P)|}))+\omega(P)-1],
\end{equation*}
and
\begin{equation*}
|\insstar(D)|=\prod_{P\in\mathcal{A}}\widetilde{\pi}_{|M(P)|}(\omega(M_{P,1}),\ldots,\omega(M_{P,|M(P)|}),\epsilon(M_{P,1}),\ldots,\epsilon(M_{P,|M(P)|})).
\end{equation*}
\end{prop}
\begin{proof}
For every $P\in\mathcal{A}$, let $T(P):=\bigcap\{D_M\mid M\in M(P)\}$. Then, $\{T(P)\mid P\in\mathcal{A}\}$ is the standard decomposition of $D$; hence, $|\insfstar(D)|=\prod\{|\insfstar(T(P))|:P\in\mathcal{A}\}$, and likewise for $|\insstar(D)|$.

By Proposition \ref{prop:cutting}, for each $P$ the set $\insfstar(T(P))$ is equal to the ordinal sum of $\inssemistar(T/P)$ and $\insfstar(T(P)_{PT_P})\setminus\{d\}$; the cardinality of the former is $\pi_{|M(P)|}(M_{P,1},\ldots,M_{P,|M(P)|})$ (by Theorem \ref{teor:num-semistar}) while the cardinality of the latter is $\omega(P)-1$, since $T(P)_{PT_P}=Z(P)$. The first claim follows.

Analogously, $\insstar(T(P))$ corresponds bijectively to $\inssmstar(T/P)$, whose cardinality is given by $\widetilde{\pi}_n$ (by Theorem \ref{teor:num-(semi)star}). The second claim follows.
\end{proof}

We end the paper by calculating two of the polynomials $\pi_\mathcal{P}$ and $\widetilde{\pi}_\mathcal{P}$.
\begin{ex}\label{ex:pi2}
The calculation of $\pi_2$ and $\widetilde{\pi}_2$.

Let $D$ be a semilocal Pr\"ufer domain with $\hiSpec(D)=\{(0),M,N\}$. Then, $\SkOver(D)=\{D,D_M,D_N,K\}$; let $\Delta\subseteq\SkOver(D)$ be a possible support for a semistar operation on $D$. Then, $K\in\Delta$, and if $D_M,D_N\in\Delta$ then also $D\in\Delta$, Hence, there are seven acceptable $\Delta$.

\begin{description}
\item[$\Delta=\{K\}$] In this case, $\Delta(M)=\Delta(N)=\emptyset$, and we have a single semistar operation.

\item[$\Delta=\{D,K\}$] In this case, $\Delta(M)=\Delta(N)=\{D\}$ are both isomorphic to $\underline{1}$.

\item[$\Delta=\{D_M,K\}$] In this case, $\Delta(M)=\{D_M\}\simeq\underline{1}$ while $\Delta(N)=\emptyset$.

\item[$\Delta=\{D_N,K\}$] Symmetrically, $\Delta(M)=\emptyset$ while $\Delta(N)=\{D_N\}\simeq\underline{1}$.

\item[$\Delta=\{D,D_M,K\}$] In this case, $\Delta(M)=\{D,D_M\}\simeq \underline{2}$ while $\Delta(N)=\{D\}\simeq\underline{1}$.

\item[$\Delta=\{D,D_N,K\}$] Symmetrically, $\Delta(M)=\{D\}\simeq\underline{1}$ while $\Delta(N)=\{D,D_N\}\simeq\underline{2}$.

\item[$\Delta=\{D,D_M,D_N,K\}$] In this case, $\Delta(M)=\{D,D_M\}\simeq \underline{2}$ and $\Delta(M)=\{D,D_N\}\simeq\underline{2}$.
\end{description}

Let now $a:=\omega(M)$ and $b:=\omega(N)$. Adding all the cases, $\insstar(D)$ is equal to
\begin{equation*}
\begin{split}
& 1+H_1(a)H_1(b)+H_1(a)+H_1(b)+H_2(a)H_1(b)+H_1(a)H_2(b)+H_2(a)H_2(b)=\\
& =1+ab+a+b+\inv{2}a(a+1)b+\inv{2}ab(b+1)+\inv{4}a(a+1)b(b+1)=\\
& =1+a+b+\frac{9}{4}ab+\frac{3}{4}(a^2b+ab^2)+\inv{4}a^2b^2
\end{split}
\end{equation*}
and the last line represents exactly $\pi_2(a,b)$.

For the (semi)star operations, we must not consider the supports $\{K\}$, $\{D_M,K\}$ and $\{D_N,K\}$. Let $\epsilon_1:=\epsilon(M)$ and $\epsilon_2:=\epsilon(N)$.

The possible $\Delta(\cdot)$ are, as above, $\underline{1}$ and $\underline{2}$; in the former case, we have $H_1'(n,\epsilon)=\epsilon$ possibilities, while in the latter we have, following the proof of Theorem \ref{teor:(semi)star},
\begin{equation*}
H'_2(n,\epsilon)=H_1(n)+(\epsilon-1)(H_1(n-1))=n+(\epsilon-1)(n-1)=\epsilon n-\epsilon+1.
\end{equation*}

Thus the cardinality of $\inssmstar(D)$ is equal to:
\begin{equation*}
\begin{split}
& H'_1(a,\epsilon_1)H'_2(b,\epsilon_2)+H'_2(a,\epsilon_1)H'_1(b,\epsilon_2)+H'_1(a,\epsilon_1)H'_2(b,\epsilon_2)+H'_2(a,\epsilon)+H'_2(b,\epsilon)=\\
&=\epsilon_1\epsilon_2+(\epsilon_1 a-\epsilon_1+1)\epsilon_2+\epsilon_1(\epsilon_2 b-\epsilon_2+1)+(\epsilon_1 a-\epsilon_1+1)(\epsilon_2 b-\epsilon_2+1)=\\
&=(1+\epsilon_1a)(1+\epsilon_2b),
\end{split}
\end{equation*}
i.e., $\widetilde{\pi}_2(a,b,\epsilon_1,\epsilon_2)=(1+\epsilon_1a)(1+\epsilon_2b)$.

Using Proposition \ref{prop:num-star-fstar}, this provides a different proof of \cite[Theorem 4.3]{hmp-overprufer}. Indeed, suppose that $A$ is a Pr\"ufer domain with Y-shaped spectrum: that is, suppose that $\Max(A)=\{M_1,M_2\}$ and that the largest prime ideal in $M_1\cap M_2$ is $P\neq 0$. Under these hypothesis, using the previous calculation,
\begin{equation*}
|\insstar(A)|=(1+\epsilon(M_1)\omega(M_1))(1+\epsilon(M_2)\omega(M_2)).
\end{equation*}
In the notation of \cite[Theorem 4.3]{hmp-overprufer}, let $m_i$ (respectively, $n_i$) be the number of nonidempotent (respectively, idempotent) prime ideals strictly between $M_i$ and $P$. Then, $\omega(M_i)=m_i+2n_i+\epsilon(M_i)$; substituting this expression in the previous one, and considering the cases $\epsilon(M_i)=1$ and $\epsilon(M_i)=2$, we obtain exactly the statement of \cite[Theorem 4.3]{hmp-overprufer}.
\end{ex}

\begin{oss}
The previous example shows that $\widetilde{\pi}_2$ splits nicely into two factors, each one containing quantities relative to a single maximal ideal. This is most likely a phenomenon restricted to the case $n=2$. Indeed, by \cite[Theorem 4.6]{hmp-overprufer}, $\widetilde{\pi}_3(1,1,1,1,1,1)=45$; if $\widetilde{\pi}_3$ would have three factors, each one relative to one maximal ideal, by symmetry we should expect 45 to be the cube of a rational number, and this is clearly not the case.

It is also possible to repeat the calculation of Example \ref{ex:pi2} for three maximal ideals; the resulting polynomials $\pi_3$ and $\widetilde{\pi}_3$ turn out to be several lines long.
\end{oss}

\begin{ex}
Let $D$ be a Pr\"ufer domain such that $\hiSpec(D)$ is the following set:
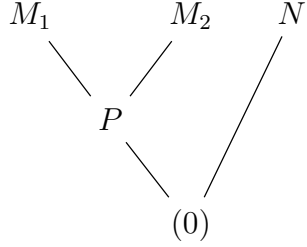
\begin{equation*}
\begin{tikzcd}[every arrow/.append style={-},column sep=tiny]
M_1\arrow{dr} & & \arrow{dl}M_2 & & N\arrow{ddll}\\
& P\arrow{dr} & & \\
& & (0)
\end{tikzcd}
\end{equation*}
Suppose $\omega(M_1)=\omega(M_2)=1$ and let $\omega(P)=a$, $\omega(N)=b$. We want to calculate $|\inssemistar(D)|$.

We have $\Theta:=\{D_N,D_{\{P\}}\}$, where $D_{\{P\}}:=D_{M_1}\cap D_{M_2}$. As in the previous example, we obtain
\begin{equation*}
|\inssemistar(D)|=1+R_1(a)H_1(b)+R_1(a)+H_1(b)+R_2(a)H_1(b)+R_1(a)H_2(b)+R_2(a)H_2(b),
\end{equation*}
where $R_1(a)$ and $R_2(a)$ denotes the number of order-preserving maps from (respectively) $\underline{1}$ and $\underline{2}$ to $\insfstar(D_{\{P\}})$.

Let $A:=D_{\{P\}}/PD_{\{P\}}$, and let $k$ be its quotient field. By Proposition \ref{prop:cutting}, there is a bijection $\insfstar(D_{\{P\}})\leftrightarrow\inssemistar(A)\oplus(\insfstar(D_P)\setminus\{d\})$. Since $\omega(M_1)=\omega(M_2)=1$, the set $\inssemistar(A)$ corresponds to the subsets of $\SkOver(A)\setminus\{k\}$ that are closed by intersections; if $Z$ and $W$ are the maximal ideals of $A$, we have seven possibilities, namely $\emptyset$, $\{A\}$, $\{A_Z\}$, $\{A_W\}$, $\{A,A_Z\}$, $\{A,A_W\}$ and $\{A,A_Z,A_W\}$. Hence, the order on $\inssemistar(A)\setminus\{\wedge_{\{k\}}\}$ corresponds to the following:
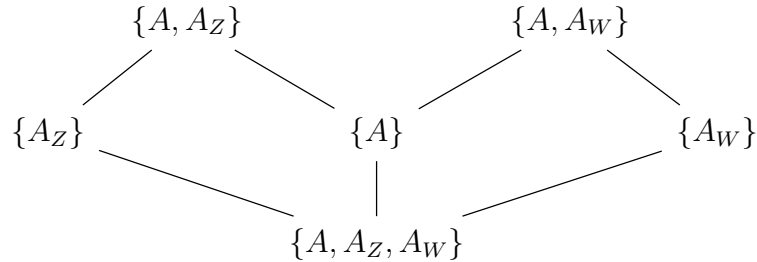
\begin{equation*}
\begin{tikzcd}[every arrow/.append style={-},column sep=tiny]
& \{A,A_Z\}\arrow{dl}\arrow{dr} & & \{A,A_W\}\arrow{dl}\arrow{dr}\\
\{A_Z\}\arrow{drr} & & \{A\}\arrow{d} & & \{A_W\}\arrow{dll}\\
& & \{A,A_Z,A_W\}\\
\end{tikzcd}
\end{equation*}

It follows that $R_1(a)=6+a$, while
\begin{equation*}
R_2(a)=15+6a+\frac{a(a+1)}{2}=\inv{2}a^2+\frac{13}{2}a+15,
\end{equation*}
and thus (at the end of the calculation) we have
\begin{equation*}
|\inssemistar(D)|=\inv{4}a^2b^2+\frac{3}{4}a^2b+\frac{15}{4}ab^2+\frac{21}{2}b^2+ \frac{45}{4}ab+a+\frac{65}{2}b+7.
\end{equation*}
\end{ex}

\end{document}